\documentclass[11pt]{amsart}

\usepackage{amscd,amssymb,mathrsfs}
\usepackage{enumerate}
\usepackage{xcolor}
\usepackage{tikz-cd}
\usepackage{tikz}
\usetikzlibrary{matrix}

\usepackage[hidelinks]{hyperref}

\newcommand{\Z}{\mathbb Z}
\newcommand{\Q}{\mathbb Q}
\newcommand{\R}{\mathbb R}
\newcommand{\C}{\mathbb C}
\renewcommand{\P}{\mathbb P}

\newcommand{\bD}{\mathbb D}
\newcommand{\bE}{\mathbb E}
\newcommand{\bF}{\mathbb F}
\newcommand{\bG}{\mathbb G}
\newcommand{\bH}{\mathbb H}
\newcommand{\bK}{\mathbb K}
\newcommand{\bL}{\mathbb L}
\newcommand{\bM}{\mathbb M}
\newcommand{\bS}{\mathbb S}
\newcommand{\bT}{\mathbb T}

\newcommand{\bV}{\mathbb V}

\newcommand\kk{{\Bbbk}}

\newcommand{\cA}{\mathcal A}

\newcommand{\cC}{\mathcal C}
\newcommand{\cE}{\mathcal E}

\newcommand{\cH}{\mathcal H}

\newcommand{\cJ}{\mathcal J}
\newcommand{\cK}{\mathcal K}

\newcommand{\cM}{\mathcal M}

\newcommand{\cO}{\mathcal O}
\newcommand{\cT}{\mathcal T}
\newcommand{\cV}{\mathcal V}
\newcommand{\cZ}{\mathcal Z}

\newcommand{\sA}{\mathscr A}
\newcommand{\sB}{\mathscr B}

\newcommand{\sE}{\mathscr E}
\newcommand{\sF}{\mathscr F}

\newcommand{\sW}{\mathscr W}

\newcommand{\fd}{\mathfrak{d}}
\newcommand{\fa}{\mathfrak{a}}
\newcommand{\g}{\mathfrak{g}}
\newcommand{\h}{\mathfrak{h}}
\newcommand{\p}{\mathfrak{p}}

\renewcommand{\sp}{\mathfrak{sp}}

\newcommand{\G}{\Gamma}

\newcommand{\ba}{\mathbf{a}}
\newcommand{\bb}{\mathbf{b}}
\newcommand{\ee}{\mathbf{e}}
\newcommand{\bv}{\mathbf{v}}

\newcommand{\deltabar}{\overline{\delta}}
\newcommand{\deltatilde}{\widetilde{\delta}}
\newcommand{\nubar}{\overline{\nu}}
\newcommand{\nutilde}{\tilde{\nu}}
\newcommand{\nuhat}{\widehat{\nu}}

\newcommand{\Cbar}{\overline{C}}
\newcommand{\cMbar}{\overline{\cM}}
\newcommand{\Qbar}{\overline{\Q}}
\newcommand{\Sbar}{{\overline{S}}}

\newcommand{\Jtilde}{\widehat{J}}	

\newcommand{\vv}{{\vec v}}

\newcommand{\del}{\partial}

\newcommand{\nablabar}{\overline{\nabla}}

\newcommand{\Sp}{\mathrm{Sp}}
\newcommand{\SL}{\mathrm{SL}}
\newcommand{\GL}{\mathrm{GL}}
\newcommand{\PSL}{\mathrm{PSL}}

\newcommand{\Gm}{\mathrm{\mathbb{G}_m}}

\newcommand{\MHS}{\mathsf{MHS}}

\newcommand{\sing}{\mathrm{sing}}

\newcommand{\nilp}{\mathrm{nil}}
\newcommand{\red}{\mathrm{red}}
\newcommand{\hyp}{\mathrm{hyp}}
\newcommand{\vol}{\mathrm{vol}}
\newcommand{\tor}{\mathrm{tor}}

\newcommand{\ac}{I}		

\newcommand{\basis}{\ba_1,\dots,\ba_g,\bb_1,\dots,\bb_g}

\newcommand{\bdot}{\bullet}

\newcommand{\bs}{\backslash}

\DeclareFontFamily{U}{MnSymbolC}{}
\DeclareSymbolFont{MnSyC}{U}{MnSymbolC}{m}{n}
\DeclareFontShape{U}{MnSymbolC}{m}{n}{
    <-6>  MnSymbolC5
   <6-7>  MnSymbolC6
   <7-8>  MnSymbolC7
   <8-9>  MnSymbolC8
   <9-10> MnSymbolC9
  <10-12> MnSymbolC10
  <12->   MnSymbolC12}{}
\DeclareMathSymbol{\intprod}{\mathbin}{MnSyC}{'270}

\newcommand\im{\operatorname{im}}               
\newcommand\id{\operatorname{id}}

\newcommand\rank{\operatorname{rk}}
\newcommand\Hom{\operatorname{Hom}}
\newcommand\End{\operatorname{End}}
\newcommand\Ext{\operatorname{Ext}}
\newcommand\Aut{\operatorname{Aut}}

\newcommand\Jac{\operatorname{Jac}}

\newcommand\Gr{\operatorname{Gr}}
\newcommand\Res{\operatorname{Res}}

\newcommand\diag{\operatorname{diag}}

\renewcommand\div{\operatorname{div}}
\renewcommand\Im{\operatorname{Im}}

\newtheorem{theorem}{Theorem}[section]
\newtheorem{lemma}[theorem]{Lemma}
\newtheorem{proposition}[theorem]{Proposition}
\newtheorem{corollary}[theorem]{Corollary}
\newtheorem{bigtheorem}{Theorem}

\newtheorem{bigcorollary}[bigtheorem]{Corollary}

\theoremstyle{definition}
\newtheorem{definition}[theorem]{Definition}

\theoremstyle{remark}
\newtheorem{remark}[theorem]{Remark}

\begin{document}

\title[The Rank of the Normal Function of the Ceresa Cycle]
{The Rank of the Normal Functions of the Ceresa and Gross--Schoen Cycles}

\author{Richard Hain}
\address{Department of Mathematics\\ Duke University\\ Durham, NC 27708-0320}
\email{hain@math.duke.edu}

\thanks{ORCID iD: {\sf 0000-0002-7009-6971}}

\date{\today}



\maketitle

\tableofcontents

\section{Introduction}

Suppose that $C$ is a smooth complex projective curve of genus $g\ge 2$. For each point $x$ of $C$, the Abel--Jacobi mapping
\begin{equation}
\label{eqn:abel-jacobi}
\mu_x : C \to \Jac C,\quad p \mapsto [p]-[x]
\end{equation}
embeds $C$ into its jacobian. Its image is an algebraic 1-cycle in $\Jac C$ that we denote by $C_x$. Its image under the involution
$$
\iota : \Jac C \to \Jac C,\quad u\mapsto -u
$$
is another algebraic 1-cycle, $\iota_\ast C_x$, that we denote by $C_x^-$. The {\em Ceresa cycle} associated to $(C,x)$ is the algebraic 1-cycle
$$
Z_{C,x} := [C_x] - [C_x^-]
$$
in $\Jac C$. Ceresa \cite{ceresa} proved that when $C$ is a general curve of genus $g\ge 3$, the Ceresa cycle is not algebraically equivalent to 0.

A standard approach to studying the Ceresa cycle is to let $C$ vary in the moduli space $\cM_g$ of smooth projective curves of genus $g$ and to study $\nu$, the associated {\em normal function}
$$
\begin{tikzcd}
J(\Lambda^3_0 \bH) \ar[r] & \cM_g \ar[l,bend right,"\nu"'].
\end{tikzcd}
$$
Here $\bH$ denotes the variation of Hodge structure over $\cM_g$ whose fiber over the moduli point of $C$ is $H_1(C;\Z)$ and $J(\Lambda^3_0 \bH)$ denotes the family of intermediate jacobians of the primitive degree three homology of $\Jac C$:
$$
PH_3(\Jac C) := H_3(\Jac C)/\big([C]\times H_1(\Jac C)\big),
$$
where $\times$ denotes the Pontryagin product.

The {\em rank} of a normal function is a measure of its non-triviality. It is defined in detail in Section~\ref{sec:rank} and is closely related to Griffiths' infinitesimal invariant of a normal function. Briefly, if one ignores the complex structure, every family of intermediate jacobians is naturally a locally constant family of real tori. This means that, locally, a normal function can be regarded as a smooth function to a real torus. Its derivative is $\C$-linear. The rank of $\nu$ is the value of the complex rank of this function at a general curve. When $g=2$, $\nu$ is identically 0 and thus has rank 0.

\begin{bigtheorem}
\label{thm:max-rk}
For all $g\ge 3$, the normal function $\nu$ of the Ceresa cycle has the maximum possible rank, namely $\dim \cM_g$.
\end{bigtheorem}

For each pointed curve $(C,x)$ there is also the Gross--Schoen cycle \cite{gross-schoen}. It is a homologically trivial algebraic 1-cycle in $C^3$. Its normal function is 3 times the normal function of the Ceresa cycle. Ziyang Gao and Shou-Wu Zhang \cite{gao-zhang} have proved a stronger version of Theorem~\ref{thm:max-rk} which they use to prove that, for all $g\ge 3$, there is a non-empty Zariski open subset of $\cM_g/\Q$ on which the Bloch--Beilinson height of the Gross--Schoen cycle has the Northcott property. Our proof of the theorem is very different from theirs. Whereas they use Ax--Schanuel, we proceed by studying the behaviour of the Ceresa normal function near the boundary of $\cM_g$ using its associated monodromy representation.

Theorem~\ref{thm:max-rk} implies that there is a proper real analytic subvariety $\Sigma$ of $\cM_g$ such that $\nu$ has maximal rank on the complement $U$ of $\Sigma$. As Gao and Zhang observe via a spreading argument \cite[Prop.~C.7]{gao-zhang}, Theorem~\ref{thm:max-rk} implies that if $C$ is a smooth projective curve whose moduli point lies in $U\setminus \cM_g(\Qbar)$, then the Ceresa and Gross--Schoen cycles associated to $C$ both have infinite order mod rational equivalence. This is a weaker version of their \cite[Thm.~1.3]{gao-zhang}.

The theorem is proved by induction on the genus. The base case $g=3$ is proved in Section~\ref{sec:genus3proof} using a result \cite{collino-pirola} of Collino and Pirola. (See Theorem~\ref{thm:collino-pirola}.) The result is not as explicit as we had hoped. However, as a consolation prize, we obtain an explicit formula for the Green--Griffiths invariant of the genus 3 Ceresa cycle as a Teichm\"uller modular form, as well as a new proof of the theorem of Collino and Pirola. The inductive step is proved by studying the behaviour of the Ceresa normal function in a neighbourhood of the boundary divisor $\Delta_0$ in the Deligne--Mumford compactification $\cMbar_g$ of $\cM_g$ by computing the monodromy representation of the restriction of the Ceresa normal function to this neighbourhood.

\begin{bigtheorem}
\label{thm:gg-genus3}
In genus 3, the Green--Griffiths invariant of the Ceresa cycle is a non-zero multiple of the Teichm\"uller modular form $\chi_{4,0,-1}$. Its restriction to the hyperelliptic locus is a non-zero multiple of the restriction of the Siegel modular form $\chi_{4,0,8}$ to the hyperelliptic locus.
\end{bigtheorem}

For a more precise statement, see Theorem~\ref{thm:gg-genus3-final}. It would be interesting to compute the multiple, which one might expect to lie in $\Q^\times$.

Since the restriction of $\chi_{4,0,8}$ to the hyperelliptic locus has no zeros \cite[Prop.~6.5]{vdg-kouvidakis}, we obtain the following strengthening of the genus 3 case of a result \cite[Thm.~6.5]{harris} of Bruno Harris.

\begin{bigcorollary}
\label{cor:harris}
In genus 3, the rank of the Ceresa normal function is exactly 1 at every point of the hyperelliptic locus.
\end{bigcorollary}

The proof of Theorem~\ref{thm:gg-genus3} makes essential use of the extension \cite{cfg} by Cl\'ery, Faber and van der Geer of Ichikawa's theory \cite{ichikawa} of Teichm\"uller modular forms in genus 3. It also gives a geometric interpretation of the meromorphic Teichm\"uller modular form $\chi_{4,0,-1}$, which has a simple pole along the hyperelliptic locus, as a multiple of the Green--Griffiths invariant of the normal function of the Ceresa cycle. Although it is a meromorphic section of an automorphic vector bundle, it is holomorphic when considered as a section of the tensor product of $\Omega_{\cM_3}^1$ with an automorphic vector bundle as we explain in detail in Section~\ref{sec:proof-cp}. This suggests that the definition of vector-valued Teichm\"uller modular forms should be enlarged to include sections of tensor products of automorphic vector bundles with the logarithmic de~Rham complex of moduli spaces of curves.

As a corollary of the proof of Theorem~\ref{thm:gg-genus3} we obtain a new proof of the result of Collino and Pirola used in the proof of Theorem~\ref{thm:max-rk} and the strengthening of Harris's result stated above. These results suggest that the Green--Griffiths invariant of the Ceresa normal function should be a Teichm\"uller modular form for all genera.

The idea of the inductive step in the proof of Theorem~\ref{thm:max-rk} is to study the behaviour of the Ceresa normal function $\nu$ in a neighbourhood of the boundary divisor $\Delta_0$ of the Deligne--Mumford moduli space $\cMbar_g$. For this we introduce the notion of {\em the residual normal function} associated to a normal function defined on the complement of a divisor. It is a normal function defined on the smooth locus of the divisor. In the case of the Ceresa normal function, the residual normal function is defined on the smooth locus of $\Delta_0$. We are able to compute it by computing its monodromy representation and show, using the inductive hypothesis, that it has rank $3g-4$. To complete the proof, we show that the normal rank of $\nu$ along the smooth locus of $\Delta_0$ is 1 using a result \cite{hain:biext} about real biextensions and a monodromy computation.

\subsection{Overview}

The proof of Theorem~\ref{thm:max-rk} requires an understanding of the limiting behaviour of admissible variations of MHS (as does the proof of Gao and Zhang). It also requires some background in the topology of moduli spaces of curves, such as the fact that normal functions over them are determined by their monodromy representations when the genus is at least 3. For this reason, in an attempt to make the paper accessible to a reasonably broad audience, we have included background material on both of these topics.

Part~\ref{part:prelims} is a brief review of the Griffiths infinitesimal invariant $\delta(\nu)$ of a normal function $\nu$ and Green's refinement $\deltabar(\nu)$ of it. Both vary holomorphically. In general, the derivative $\nabla \nu_\R$ of the normal function varies real analytically, but not holomorphically. This is the case with the Ceresa normal function. Since $\nabla \nu_\R$ is real, it is determined by its $(1,0)$ component, which we call the {\em canonical derivative} of $\nu$ and denote by $\nablabar \nu$. It determines both the rank of $\nu$ and its Green--Griffiths invariant. It plays a fundamental role in this paper. An obstruction to obtaining an explicit formula for the rank of the Ceresa normal function at a genus 3 curve is that $\nablabar \nu$ does not vary holomorphically.

In Part~\ref{part:ceresa} we present general results about the Green--Griffiths invariant and canonical derivative of the Ceresa normal function. Part~\ref{part:genus3} focuses on the genus 3 case. One key observation is that in genus 3 the canonical derivative $\nablabar \nu$ can be decomposed naturally
$$
\nablabar \nu = \deltabar(\nu) + \nablabar f
$$
where $f$ is real analytic and $\deltabar(\nu)$ is the Green--Griffiths invariant. Away from the hyperelliptic locus, the two components are sections of different automorphic vector bundles. At the moduli point of a non-hyperelliptic curve $C$, both components can be interpreted as quadratic forms on the tangent space of $\cM_3$ at the moduli point of $C$. Using the fact that both components are invariant under $\Aut C$, we show that $\nablabar f$ vanishes at the Klein quartic and, using the Collino--Pirola Theorem, that the other component has maximal rank. This establishes the base case of Theorem~\ref{thm:max-rk}. This part concludes with the proof of Theorem~\ref{thm:gg-genus3} and related results.

Admissible normal functions are, by definition \cite{hain:msri,saito}, period mappings of certain variations of mixed Hodge structure. Normal functions associated to homologically trivial algebraic cycles are admissible. This imposes strong conditions on their boundary behaviour. These restrictions play a key role in the inductive step of the proof of Theorem~\ref{thm:max-rk} and also in \cite{gao-zhang}. Part~\ref{part:vmhs} contains a review of admissible variations of MHS. We give a quick construction of the N\'eron model of a family of intermediate jacobians in the special case where the variation is a family of 1-dimensional nilpotent orbits. We also introduce the notion of the {\em residual normal function} $\nu_\Delta$ associated to an admissible normal function $\nu$ defined on $S-\Delta$, where $\Delta$ is a smooth divisor in the smooth variety $S$, and explain how its rank is related to that of $\nu$.

In Part~\ref{part:higher-genus} we apply the results in Part~\ref{part:vmhs} to prove the inductive step. We use the fact that normal functions over $\cM_{h,n}$ associated to variations corresponding to irreducible representations of $\Sp_h$ are determined by their monodromy representations when $h \ge 3$. (See Appendix~\ref{app:normal}.) This allows us to compute the monodromy of the Ceresa normal function in a neighbourhood of the smooth locus $\Delta$ of the boundary divisor $\Delta_0$ of $\cMbar_g$ as well as the residual normal function $\nu_\Delta$. The inductive hypothesis then implies that $\nu_\Delta$ has maximal rank $3g-4$. A result \cite{hain:biext} about real biextensions associated to curves then implies that the ``normal rank'' of $\nu$ along $\Delta$ is 1, which establishes the result.

\subsection{Conventions}

All algebraic varieties and stacks will be defined over the complex numbers. The moduli spaces of curves $\cM_g$ and of principally polarized abelian varieties $\cA_g$ will be regarded as stacks.

\bigskip

\noindent{\bf Acknowledgments:} I am especially grateful to Shou-Wu Zhang for asking me about the rank of the Ceresa normal function. Without his interest, this paper would never have seen the light of day. I am also grateful to Gerard van der Geer for bringing his work \cite{vdg-kouvidakis} with Kouvidakis to my attention. This yielded Corollary~\ref{cor:harris}. I am also indebted to Haohua Deng for pointing out several errors, now fixed, in Part~\ref{part:higher-genus} of an early draft of this paper and to the referees whose numerous helpful comments and corrections resulted in significant improvements to the paper. Finally, I would like to thank the Simons Foundation for travel support.

\part{Preliminaries}
\label{part:prelims}

This part is a review of some basic facts about normal functions, such as the construction of normal functions associated to families of homologically trivial cycles and the construction of the Griffiths and Green--Griffiths invariants of normal functions. The canonical derivative and the rank of a normal function are defined in Section~\ref{sec:rank}.

\section{Normal functions and families of algebraic cycles}

In this section, we review and elaborate on work of Griffiths \cite{griffiths} and Green \cite{green} on invariants of normal functions. We also recall the definition of the canonical foliation of a family of intermediate jacobians associated to a variation of Hodge structure of weight $-1$.

\subsection{The MHS associated to a homologically trivial cycle}
\label{sec:ext_cycle}

Suppose that $X$ is a smooth projective variety and that that $Z= \sum_j n_j Z_j$ is an algebraic $d$-cycle on $X$ where the $Z_j$ are distinct reduced irreducible subschemes of $X$ of dimension $d$. Denote the support of $Z$ by $|Z|$. When $Z$ is homologically trivial, it determines an extension
$$
\begin{tikzcd}[column sep=small]
0 \ar[r] & H_{2d+1}(X)(-d) \ar[r] & E_Z \ar[r] & \Z(0) \ar[r] & 0
\end{tikzcd}
$$
of mixed Hodge structures. It is obtained from the long exact homology sequence of $(X,|Z|)$ by pulling back along the map
$$
cl_Z : \Z(d) \to H_{2d}(|Z|) = \textstyle{\bigoplus_j} \Z[Z_j]
$$
that takes $1$ to $[Z] = \sum n_j[Z_j]$ and then twisting by $\Z(-d)$:
$$
\begin{tikzcd}[column sep=small]
0 \ar[r] & H_{2d+1}(X) \ar[r] & H_{2d+1}(X,|Z|) \ar[r] & H_{2d}(|Z|) \ar[r] & H_{2d}(X) \\
0 \ar[r] & H_{2d+1}(X) \ar[u,equal] \ar[r] & E_Z(d) \ar[u,hookrightarrow] \ar[r] & \Z(d) \ar[u,"cl_Z"'] \ar[r] & 0 
\end{tikzcd}
$$
The extension depends only on the rational equivalence class of $Z$. The extension $(E_Z)_\Z$ is generated by $H_{2d+1}(X;\Z)$ and $\G$, where $\partial \G = Z$.

\subsection{Extensions of mixed Hodge structure and intermediate jacobians}

Suppose that $V$ is a Hodge structure of negative weight. The group of extensions of mixed Hodge structure (MHS) of the form
$$
0 \to V \to E \to \Z \to 0
$$
forms a group $\Ext^1_\MHS(\Z,V)$ that is isomorphic to the complex torus
$$
J(V) := V_\C/(V_\Z + F^0 V).
$$
The extension determines, and is determined by, the image of
$$
\ee_\Z - \ee_F \in V_\C
$$
in $J(V)$, where $\ee_\Z \in E_\Z$ and $\ee_F \in F^0 E$ both project to $1\in\Z$.

The torus $J(V)$ is compact when $V$ has weight $-1$. This is because, in this case,
$$
V_\C = F^0 V \oplus \overline{F^0 V}.
$$
This implies that the composite
$$
V_\R \hookrightarrow V_\C \to V_\C/F^0 V
$$
is an isomorphism of real vector spaces which, in turn, induces an isomorphism of $J(V)$ with the compact real torus $V_\R/V_\Z$.

So a homologically trivial $d$-cycle $Z$ in a smooth projective variety $X$ determines a point $\nu_Z$ in the Griffiths intermediate jacobian
$$
J(H_{2d+1}(X)(-d)).
$$
$$
\textstyle{\int_\G} \in \Hom_\C(F^{d+1} H^{2d+1}(X),\C) \cong H_{2d+1}(X)(-d)/F^0
$$
where, as above, $\partial \G = Z$.

\begin{remark}
\label{rem:symmetry}
If a group $G$ acts on $X$, then it also acts on the intermediate jacobian $J(H_{2d+1}(X)(-d))$. For each $\gamma \in G$, the point $\gamma_\ast \nu_Z$ is the point determined by the cycle $\gamma_\ast Z$ in $X$. So if $\gamma_\ast Z$ is rationally equivalent to $Z$ for all $\gamma \in G$, then $\nu_Z$ is fixed by $G$.
\end{remark}

\subsection{Normal functions}

Suppose $\bV$ is a variation of Hodge structure of negative weight over a smooth variety $S$. Set
$$
\cV = \bV_\Z\otimes_\Z \cO_S.
$$
It is a flat holomorphic vector bundle over $S$ with connection
$$
\nabla : \cV \to \cV \otimes \Omega^1_S.
$$
Denote the $p$th term of its Hodge filtration by $F^p\cV$. The connection satisfies Griffiths transversality
$$
\nabla : F^p \cV \to F^{p-1}\cV \otimes \Omega^1_S.
$$

To such a variation, we can associate a family of intermediate jacobians $J(\bV)$ over $S$. It is the quotient of the vector bundle $\cV/F^0\cV$ by the image of $\bV_\Z \to \cV/F^0\cV$. Its fiber over $s \in S$ is $J(V_s)$, the intermediate jacobian associated to the fiber $V_s$ of $\bV$ over $s$.

Each extension
\begin{equation}
\label{eqn:extn}
0 \to \bV \to \bE \to \Z_S \to 0
\end{equation}
of variations of mixed Hodge structure gives rise to a holomorphic section $\nu_\bE$ of the bundle
$$
J(\bV) \to S
$$
of intermediate jacobians. Its value over $s\in S$ is the point of $J(V_s)$ that corresponds to the extension
$$
0 \to V_s \to E_s \to \Z \to 0
$$
of MHS obtained by restricting (\ref{eqn:extn}) to $s$.

Every continuous section $\sigma$ of $J(\bV)$ determines a class $c(\sigma)$ in $H^1(S,\bV_\Z)$. A detailed description of $c(\sigma)$ can be found in \cite[\S4.1]{hain:normal}. Since $H^1(S,\bV_\Z)$ is the group of congruence classes of extensions of $\Z_S$ by $\bV_\Z$ in the category of local systems over $S$, each section $\sigma$ of $J(\bV)$ determines an extension of local systems
$$
0 \to \bV \to \bE \to \Z_S \to 0
$$
over $S$. The value of $\sigma$ at $s\in S$ determines a MHS on the fiber of $\bE$ over $s$.

\begin{definition}
A section of $J(\bV) \to S$ is a {\em normal function} if the corresponding extension is a variation of MHS.
\end{definition}

In other words, the group of normal function sections of $J(\bV)$ is isomorphic to the group of extensions of $\Z_S$ by $\bV$ in the category of variations of MHS over $S$. Normal functions have to be holomorphic. This corresponds to the condition that the Hodge filtration varies holomorphically. They also have to satisfy the {\em Griffiths infinitesimal period relation}, which is equivalent to the condition that the canonical flat connection on
$$
\cE := \bE\otimes_\Z \cO_S
$$
satisfies Griffiths transversality, $\nabla : F^p \cE \to F^{p-1} \cE \otimes \Omega^1_S$.

\begin{remark}
The variations of MHS that arise in algebraic geometry satisfy additional conditions. Such variations are called {\em admissible} variations of MHS. The normal functions that correspond to admissible variations of MHS are called {\em admissible normal functions}. The additional conditions on an admissible normal function $\nu$ restrict the asymptotic behaviour of $\nu(s)$ as $s$ approaches the boundary of $S$. More precisely the extension has to satisfy conditions which ensure that limit mixed Hodge structures exist. Admissible variations of MHS are defined in Section~\ref{sec:avmhs}.
\end{remark}

\subsection{Families of homologically trivial cycles}
\label{sec:families_cycles}
 
Suppose now that $S$ is a smooth variety and that $f : X \to S$ is a family of smooth projective varieties. Suppose that $Z$ is a relative $d$-cycle. That is, it is an algebraic cycle on $X$ whose restriction $Z_s$ to each fiber $X_s := f^{-1}(s)$ over $s\in S$ is a $d$-cycle. Suppose that each $Z_s$ is homologically trivial in $X_s$. Denote the local system over $S$ whose fiber over $s\in S$ is $H_{2d+1}(X_s)(-d)$ by $\bV$. It underlies a variation of Hodge structure of weight $-1$. The construction in Section~\ref{sec:ext_cycle} gives an extension of local systems
$$
0 \to \bV \to \bE_Z \to \Z_S \to 0
$$
and a mixed Hodge structure on the fiber $E_s$ over $s\in S$ that is an extension of $\Z$ by $V_s$. The work of Griffiths \cite{griffiths} implies that the corresponding section $\nu_Z : S \to J(\bV)$ is holomorphic and satisfies his infinitesimal period relation. It is therefore an extension of variations of MHS. The corresponding section $\nu_Z$ of $J(\bV)$ is thus a normal function. The main result of \cite{steenbrink-zucker} implies that it is admissible.

\subsection{The foliation of $J(\bV)$}
\label{sec:foliation}

Suppose now that $\bV$ has weight $-1$. This implies that the map
$$
J(\bV_\R) := \bV_\R/\bV_\Z \to J(\bV)
$$
is an isomorphism of families of tori over $S$. The family $J(\bV_\R)$ is a locally constant family of tori and is thus foliated by its locally constant leaves. Consequently, the family $J(\bV) \to S$ is a foliated family of complex tori. This foliation $\sF$ is a {\em complex} foliation, but not a {\em holomorphic} foliation. This means that each leaf of $\sF$ is a complex submanifold of $J(\bV)$ but that $\sF$ is not, in general, locally biholomorphic to a product foliation.

\begin{definition}
A section of $J(\bV)$ is {\em locally constant} if its image lies in a leaf of $\sF$ or, equivalently, it lies in the image of $H^0(S,\bV_\R) \to H^0(S,J(\bV))$.
\end{definition}

\subsection{The Griffiths infinitesimal invariant}

Griffiths' infinitesimal invariant is an obstruction to a normal function being locally constant. Suppose that $\nu$ is a normal function section of $J(\bV)$. Let
$$
0 \to \bV \to \bE \to \Z \to 0
$$
be the corresponding variation of MHS over $S$. Suppose that $s \in S$ and that $U$ is an open neighbourhood of $s$ that is biholomorphic to a complex ball. Since $U$ is Stein, the restriction of the bundles $F^p \cE$ to $U$ are holomorphically trivial, and since $U$ is contractible, the restriction of $\bE_\Z$ to $U$ is trivial. Let
$$
\ee_F \in H^0(U,F^0 \cE) \text{ and } \ee_\Z \in H^0(U,\bE_\Z)
$$
be lifts of $1 \in \Z$. Then
$$
v := \ee_\Z - \ee_F \in H^0(U,\cV)
$$
is a lift of the restriction of $\nu$ to $U$. Since $\ee_F$ is a local section of $F^0\cE$, Griffiths transversality implies that
\begin{equation}
\label{eqn:nabla_holo}
\nabla \nu \in H^0(U,F^{-1}\cV\otimes\Omega^1_S).
\end{equation}
Since $\ee_F$ is well defined mod $H^0(U,F^0\cV)$, this descends to a well defined element $\delta_U(\nu)$ of
$$
\cK(U) := \ker\{\nabla : H^0(U, F^{-1}\cV\otimes\Omega^1_S)\to H^0(U, F^{-2}\cV\otimes \Omega^2_S)\}/\nabla H^0(U,F^0\cV).
$$
The functor $U \mapsto \cK(U)$ is a presheaf on $S$.

This obstruction group has a more intrinsic description in terms of the Hodge filtration of $\cV\otimes\Omega^\bdot_S$, which is defined by
$$
F^p(\cV\otimes \Omega_S^j) := (F^{p-j}\cV) \otimes \Omega_S^j.
$$
Each $(F^p(\cV\otimes \Omega_S^\bdot),\nabla)$ is a complex of locally free complex analytic sheaves on $S$. Denote their cohomology sheaves by $\cH^\bdot(F^p(\cV\otimes \Omega_S^\bdot))$. Since $\cH^1(F^0(\cV\otimes \Omega_S^\bdot))$ is the sheafification of $\cK$, the germs of the $\delta_U(\nu)$ at each $s\in S$ define a section
$$
\delta(\nu) \in H^0(S,\cH^1(F^0(\cV\otimes\Omega_S^\bdot))).
$$
This is the Griffiths invariant of $\nu$. It vanishes if and only if $\nu$ is locally constant on $S$.

\subsection{Green's refinement of the Griffiths invariant}

Green's refinement \cite[\S1]{green} of Griffiths construction is a useful tool for understanding the Griffiths invariant. A key point is that each graded quotient
\begin{equation}
\label{eqn:Gr-Hodge}
\Gr_F^p (\cV\otimes \Omega_S^\bdot) := [0 \to \Gr^p_F \cV \to \Gr^{p-1}_F \cV \otimes \Omega^1_S \to \Gr^{p-2}_F \cV \otimes \Omega^2_S \to \cdots]
\end{equation}
of the Hodge filtration of $\cV\otimes\Omega_S^\bdot$ is a complex of vector bundles as the differential $\nablabar$ induced by the connection $\nabla$ is $\cO_S$-linear.

Denote the homology sheaves of the complex (\ref{eqn:Gr-Hodge}) by $\cH^\bdot(\Gr_F^p(\cV\otimes \Omega_S^\bdot))$. Suppose that $\nu$ is a normal function. Denote the image of $\delta(\nu)$ under the map
\begin{equation}
\label{eqn:quot}
H^0(S,F^{-1}\cV\otimes\Omega^1_S/\nabla (F^0\cV)) \to H^0(S,\cH^1(\Gr_F^0(\cV\otimes \Omega_S^\bdot)))
\end{equation}
by $\deltabar(\nu)$. In favourable situations, Green's invariant $\deltabar(\nu)$ determines the Griffiths invariant. 

\begin{proposition}[Green]
\label{prop:green}
If $\cH^1(\Gr_F^p(\cV\otimes \Omega_S^\bdot))$ vanishes for all $p>0$, then the map (\ref{eqn:quot}) is injective, so that $\delta(\nu)$ vanishes if and only if $\deltabar(\nu)$ vanishes.
\end{proposition}

\begin{proof}
Green proves this using a spectral sequence argument. Here we give an elementary direct proof. Since
$$
0 \to \Gr_F^1(\cV\otimes\Omega_S^\bdot) \to F^0(\cV\otimes\Omega_S^\bdot)/F^2 \to \Gr_F^0(\cV\otimes\Omega_S^\bdot) \to 0
$$
is exact, so is
$$
\cH^1(\Gr_F^1(\cV\otimes\Omega_S^\bdot)) \to \cH^1(F^0(\cV\otimes\Omega_S^\bdot)/F^2) \to \cH^1(\Gr_F^0(\cV\otimes\Omega_S^\bdot)).
$$
The vanishing of $\cH^1(\Gr_F^1(\cV\otimes\Omega_S^\bdot))$ implies the injectivity of
$$
\cH^1(F^0(\cV\otimes\Omega_S^\bdot)/F^2) \hookrightarrow \cH^1(\Gr_F^0(\cV\otimes\Omega_S^\bdot)).
$$
Similarly, the vanishing of $\cH^1(\Gr_F^2(\cV\otimes \Omega_S^\bdot))$ and the exactness of
$$
0 \to \Gr_F^2(\cV\otimes\Omega_S^\bdot) \to F^0(\cV\otimes\Omega_S^\bdot)/F^3 \to F^0(\cV\otimes\Omega_S^\bdot)/F^2 \to 0
$$
implies the injectivity of
$$
\cH^1(F^0(\cV\otimes\Omega_S^\bdot)/F^3) \hookrightarrow \cH^1(F^0(\cV\otimes\Omega_S^\bdot)/F^2).
$$
One continues in this manner to obtain injections
\begin{equation}
\label{eqn:stage_p}
\cH^1(F^0(\cV\otimes\Omega_S^\bdot)/F^{p+1}) \hookrightarrow \cH^1(F^0(\cV\otimes\Omega_S^\bdot)/F^p)
\end{equation}
for all $p>0$. Since
$$
F^0(\cV\otimes\Omega_S^\bdot)/F^m = F^0(\cV\otimes\Omega_S^\bdot)
$$
for $m$ sufficiently large, the composition of the injections (\ref{eqn:stage_p})
$$
\cH^1(F^0(\cV\otimes\Omega^\bdot_S)) \to \cH^1(\Gr_F^0(\cV\otimes\Omega_S^\bdot))
$$
is injective. The result follows by taking global sections.
\end{proof}

\section{The rank of a normal function}
\label{sec:rank}

As above, $\bV$ is a variation of Hodge structure of weight $-1$ over $S$ and $\nu : S \to J(\bV)$ is normal function.

\subsection{Definition of the rank}

We can regard the foliation $\sF$ of $J(\bV)$ (Section~\ref{sec:foliation}) as a smooth sub-bundle of the tangent bundle of $J(\bV)$. Since the leaves of the foliation are complex submanifolds of $J(\bV)$, we can regard $\sF$ as a complex sub-bundle of $TJ(\bV)$. The projection $J(\bV) \to S$ induces an isomorphism $\sF_v \to T_s S$ for all $v\in J(V_s)$. This gives a splitting of the sequence
$$
\begin{tikzcd}
&& \sF_v \ar[d] \ar[dr,"\cong"] \\
0 \ar[r] & V_s/F^0 \ar[r] & T_vJ(\bV) \ar[r] & T_s S \ar[r] & 0
\end{tikzcd}
$$
Let $\phi : TJ(\bV) \to \cV/F^0$ be the corresponding projection. It is $\C$-linear. Composing it with $(d\nu)_s$ gives a $\C$-linear map $\phi \circ (d\nu)_s : T_s S \to V_s/F^0 V_s$:
$$
\begin{tikzcd}
0 \ar[r] & V_s/F^0V_s \ar[r] & T_{\nu(s)}J(\bV) \ar[r] \ar[l,bend left,"\phi"] & T_sS \ar[ll,bend right,"\phi \circ (d\nu)_s"'] \ar[l,bend left,"(d\nu)_s"] \ar[r] & 0 
\end{tikzcd}
$$

\begin{definition}
The rank $\rank_s\nu$ of $\nu$ at $s\in S$ is the rank of $\phi \circ (d\nu)_s$. Define
$$
\rank_S \nu = \max_{s\in S} \rank_s \nu
$$
\end{definition}

The rank of $\nu$ vanishes if and only if its Griffiths invariant $\delta(\nu)$ vanishes. In general, the relation between the Griffiths invariant and the rank is more subtle.

\subsection{The canonical derivative and the Green--Griffiths invariant}

In order to relate the rank of a normal function to the rank of the associated invariants defined by Griffiths and Green, we need to consider real lifts of normal functions.

We will work with complexified tangent spaces as we need to compare invariants of real analytic and holomorphic lifts of a normal function. In this subsection, and only this subsection, we denote the real tangent space of a complex manifold $Y$ at $y\in Y$ by $T_y Y$. It has a canonical almost complex structure $\ac : T_y Y \to T_y Y$. Its complexification decomposes
$$
T_y Y \otimes \C =  T_y' Y \oplus T_y''Y
$$
into the $\ac=i$ and $\ac=-i$ eigenspaces. The $\ac=i$ eigenspace $T_y' Y$ is the holomorphic tangent space. An $\R$-linear map $T_y Y \to V$ into a complex vector space $V$ is complex linear if and only if the induced map $T_y Y\otimes \C \to V$ vanishes on $T_yY''$. The inclusion of $T_y Y$ into $T_y Y\otimes \C$ induces a complex linear isomorphism
$$
T_y Y \to (T_y Y\otimes \C)/T_y'' Y \cong T_y' Y.
$$

\begin{definition}
The {\em canonical derivative} $\nablabar_s \nu$ of $\nu$ at $s$ is the composite
$$
\begin{tikzcd}
T'_s S \ar[r,"(d\nu)_s"] & T_{\nu(s)}' J(\bV) \ar[r,"\phi"] & V_s/F^0.
\end{tikzcd}
$$
\end{definition}
Griffiths transversality implies that its image lies in the subspace $\Gr^{-1}_F V_s$ of $V_s/F^0$. The function $\nablabar : s \mapsto \nablabar_s\nu$ is a real analytic section of $\Gr^{-1}_F \cV$. 

Denote the sheaf of real analytic functions on $S$ by $\sE_S$. The normal function $\nu$ corresponds to a real analytic section $\nu_\R$ of $J(\bV_\R)$ under the isomorphism $J(\bV_\R) \cong J(\bV)$. Locally it lifts to a section $\nutilde_\R$ of $\bV_\R\otimes \sE_S$. Its derivative $(\nabla \nutilde_\R)_s \in V_{s,\R}\otimes T_s^\vee S$ at $s\in S$ does not depend on the choice of $\nutilde_\R$, so we will denote it by $(\nabla \nu_\R)_s$.

The goal of the next result is to clarify the relationship between the various derivatives of $\nu$.

\begin{proposition}
\label{prop:derivatives}
The canonical derivative $\nablabar_s\nu$ is related to the derivatives of $\nu$ and $\nu_\R$ at $s\in S$ by the commutative diagram
\begin{equation}
\label{eqn:derivatives}
\begin{tikzcd}[column sep=64]
T_s S \ar[r,"(d\nu_\R)_s"]  \ar[rr,bend left=24,"(\nabla \nu_\R)_s"] \ar[d,"\cong"] & T_{\nutilde_\R(s)} J(\bV_\R) \ar[d,"\cong"] \ar[r,"\phi_\R"] & V_{s,\R} \ar[d,"\cong"] \\
T'_s S \ar[r,"(d\nu)_s"] \ar[drr,bend right=16,"\nablabar_s \nu"] & T_{\nu(s)}' J(\bV)  \ar[r,"\phi"] & V_s/F^0 \\
& & \Gr_F^{-1} V_s \ar[u,hookrightarrow]
\end{tikzcd}
\end{equation}
In particular,
$$
\rank_s \nu = \rank \nablabar_s \nu = \rank_\R (\nabla \nu_\R)_s/2.
$$
\end{proposition}

\begin{proof}
This is an exercise in the definitions using the fact that $d\nu_\R$ is $\C$ linear.
\end{proof}

\begin{corollary}
\label{cor:nablabar}
The projection of the $(1,0)$ component $\nabla'\nu_\R$ of the derivative $\nabla\nu_\R$ to a section of $(\cV/F^0\cV)\otimes \sE_S\otimes \Omega^1_S$ equals $\nablabar \nu \in H^0(S,\Gr_F^{-1}\cV \otimes \Omega_S^1\otimes \sE_S)$.
\end{corollary}

It is important to note that, in general, $\nablabar \nu$ is a real analytic, but not holomorphic, section of $\Gr_F^{-1}\cV \otimes \Omega_S^1$, even though $\deltabar(\nu)$, its reduction mod $\nablabar \Gr_F^0 \cV \otimes \sE_S$, is a holomorphic section of $\Gr_F^{-1}\cV \otimes \Omega_S^1/\nablabar\Gr_F^0\cV$. This will be important in Section~\ref{sec:proof-cp} where the possibility that $\nablabar \nu$ is not holomorphic is the obstruction to obtaining an explicit formula for $\nablabar\nu$ in genus 3.

\subsection{Symmetry}
\label{sec:symmetry_families}

Suppose that $S$ is a complex analytic variety and that $\bV \to S$ is a variation of Hodge structure of weight $-1$. Suppose that $G$ is a finite group that acts holomorphically on $S$ and that this action lifts to an action on $\bV \to S$ as a variation of Hodge structure. If $s \in S$ is a fixed point of the action, then $G$ acts on $T_s S$ and as an automorphism group of the Hodge structure $V_s$ and its intermediate jacobian $J(V_s)$. If $\nu$ is a $G$-invariant normal function section of $J(\bV)$, then $\nu(s) \in J(V_s)$ will be a fixed point of the $G$ action.

\begin{lemma}
\label{lem:symmetry}
Under these assumptions, $G$ acts on the diagram (\ref{eqn:derivatives}) in Proposition~\ref{prop:derivatives}. That is, $G$ acts on each vertex and each morphism is $G$ equivariant. \qed
\end{lemma}

\subsection{The rank filtration}

We can filter $S$
$$
S = S_d \supseteq S_{d-1} \supseteq \cdots \supseteq S_0 \supseteq S_{-1} := \emptyset
$$
by the strata $S_r$ where $\nu$ has rank $\le r$, where $d=\dim S$. Since $\nu_\R$ is real analytic, each $S_r$ is a real analytic subvariety of $S$. Gao and Zhang \cite[Thm.~1.4]{gao-zhang} have proved a stronger result. They define the {\em foliation filtration} $S_\sF$ of $S$ to be the union of the leaves of the ``pullback foliation'' $\nu^\ast \sF$ of codimension $\le r$ and show that each $S_\sF(r)$ is a complex algebraic subvariety of $S$. The rank condition implies that $S_r$ contains $S_\sF(r)$ for all $r$. 

The following much weaker density result is sufficient for our purposes.

\begin{proposition}
\label{prop:density}
If there is a point $x_0 \in S$ where $\rank_{x_0} \nu = \dim S$, then
$$
\Sigma := S_{-1+\dim S} = \{x \in S : \rank_x \nu < \dim S\}
$$
is a real analytic subvariety of $S$ of real codimension $\ge 1$. In particular, $S-\Sigma$ is dense in $S$. \qed
\end{proposition}

\part{The Ceresa cycle}
\label{part:ceresa}

In this part we recall the construction of the normal function of the Ceresa cycle and discuss generalities related to its Green--Griffiths invariant and canonical derivative.

\section{The normal function of the Ceresa cycle}

Suppose that $g\ge 3$. The universal curve $f : \cC \to \cM_{g,1}$ has a tautological section that takes the moduli point $[C,x]$ of $(C,x)$ to the point $x \in C \subset \cC$. This allows us to imbed $\cC$ into the universal jacobian $\cJ := \Jac_{\cC/\cM_{g,1}}$ via a map $\mu_x : \cC \to \cJ$ whose restriction to the fiber over $[C,x]$ is the Abel--Jacobi map (\ref{eqn:abel-jacobi}):
$$
\begin{tikzcd}
\cC \ar[d,"f"] \ar[r,hookrightarrow,"\mu_x"] & \cJ \ar[d] \\
\cM_{g,1} \ar[u,bend left=36pt, "x"] \ar[r,equals] & \cM_{g,1}
\end{tikzcd}
$$
The relative Ceresa cycle $\cZ_x$ in $\cJ$ is
$$
\cZ_x := (\mu_x)_\ast \cC - \iota_\ast (\mu_x)_\ast \cC,
$$
where $\iota : u \mapsto -u$ is the involution of $\cJ$. Its restriction to the fiber over $[C,x] \in \cM_{g,1}$ is the Ceresa cycle $Z_{C,x}$ in $\Jac C$.

Set
$$
\bH = R^1 f_\ast \Z(1).
$$
Poincar\'e duality implies that it has fiber $H_1(C;\Z)$ over $[C,x]$. The variation $\Lambda^3 \bH$ has fiber $H_3(\Jac C)$ over $[C,x]$. The normal function $\nu_x$ associated to $\cZ_x$ is an admissible section of $J(\Lambda^3 \bH(-1))$ over $\cM_{g,1}$.

Pontryagin product with the class of $C$ defines an inclusion $H_1(\Jac C) \hookrightarrow H_3(\Jac C)(-1)$ of Hodge structures and therefore an inclusion $\Jac C \hookrightarrow J(H_3(\Jac C)(-1))$. Pulte \cite{pulte} showed that if $x,x' \in C$, then
$$
\nu_x(C) - \nu_{x'}(C) = 2([x] - [x']) \in \Jac C
$$
so that the image of $\nu_x(C)$ in $J(PH_3(\Jac C)(-1))$ does not depend on the choice of the base point $x\in C$. Set
$$
\Lambda^3_0 \bH (-1) = (\Lambda^3 \bH)(-1)/\bH.
$$
Since $\cJ = J(\bH)$, Pulte's result implies that the normal function $\nu_x$ descends to a normal function section of $J(\Lambda^3_0 \bH(-1))$ over $\cM_g$:
\begin{equation}
\label{eqn:nu}
\begin{tikzcd}
J(\Lambda^3_0 \bH(-1)) \ar[r] & \cM_g \ar[l,bend right,"\nu"']
\end{tikzcd}
\end{equation}

\begin{remark}
Since $\bH$ is defined on $\cM_g^c$, the normal functions $\nu_x$ and $\nu$ extend to normal functions over the moduli spaces $\cM_{g,1}^c$ and $\cM_g^c$ of curves of compact type. This follows from \cite[Thm.~7.1]{hain:msri}.
\end{remark}

We now take $X = \cM_g^c$, the moduli space of stable genus $g$ curves of compact type (or a suitable open subset), where $g \ge 3$. Let $f : \cC \to \cM_g^c$ be the universal curve. Set
$$
\bH = R^1 f_\ast \Z(1).
$$
This is the weight $-1$ polarized variation of Hodge structure (PVHS) over $\cM_g^c$ whose fiber over the moduli point $[C]$ of $C$ is
$$
H_1(C,\Z) \cong H^1(C,\Z(1)).
$$
(We will also regard this as a PVHS over $\cA_g$ and this as its pullback along the period map.) We will take $\bV$ to be the weight $-1$ PVHS
$$
\bV = (\Lambda^3 \bH/\theta\cdot \bH)(-1)
$$
where $\theta \in \Lambda^2 \bH$ is the section corresponding to the polarization. The Ceresa cycle gives a normal function section $\nu$ of $J(\bV)$ over $\cM_g^c$. We will give the Collino--Pirola formula for the Green--Griffiths invariant $\deltabar_C(\nu)$ of the Ceresa cycle at $[C]$ and a partial formula for its canonical derivative $\nablabar_C \nu$ at $[C]$.

\section{Geometry of the period mapping}
\label{sec:period-map}

Because we are computing derivatives and so working locally, it is often more convenient to work on Torelli space and Siegel space rather than on the moduli spaces $\cM_g$ and $\cA_g$.

Suppose that $g\ge 2$. Torelli space $\cT_g$ is the quotient of Teichm\"uller space by the Torelli group. It is the moduli space of pairs $(C;\basis)$ consisting of ``framed curves'' --- smooth projective curves of genus $g$ and a symplectic basis of its first integral homology group. It is a complex manifold. The symplectic group $\Sp_g(\Z)$ acts on $\cT_g$ by its action on the basis. The moduli space $\cM_g$ of smooth projective curves of genus $g$ is the stack quotient $\Sp_g(\Z)\bs \cT_g$.

The Siegel upper half plane $\h_g$ of degree $g$ is defined by
$$
\h_g = \{\tau \in \bM_g(\C) :  \tau = \tau^T \text{ and }\Im (\tau) \text{ is positive definite}\}.
$$
It is the moduli space of $g$ dimensional principally polarized abelian varieties $(A,\theta)$ of dimension $g$ together with a basis $(\ba_1,\dots,\ba_g,\bb_1,\dots,\bb_g)$ of $H_1(A;\Z)$ that is symplectic with respect to the polarization $\theta : H_1(A) \otimes H_1(A) \to \Z$. The period mapping $\cT_g \to \h_g$ takes a framed curve to its framed jacobian:
$$
(C;\basis) \to (\Jac C;\basis).
$$
It is holomorphic. The symplectic group acts on $\h_g$ via its action on the basis. The stack quotient $\Sp_g(\Z)\bs \h_g$ is the moduli space $\cA_g$ of principally polarized abelian varieties. The period map is $\Sp_g(\Z)$ equivariant and descends to a map $\cM_g \to \cA_g$.

The cotangent space of $\cT_g$ at a point corresponding to a framing of $C$ is $H^0(\Omega_C^{\otimes 2})$ and the cotangent space of its image in $\h_g$ is $S^2 H^0(\Omega_C)$. The differential is the multiplication mapping
\begin{equation}
\label{eqn:mult}
S^2 H^0(\Omega_C) \to H^0(\Omega_C^{\otimes 2}).
\end{equation}
By Noether's theorem, it is surjective when $C$ is not hyperelliptic. When $C$ is hyperelliptic, it has maximal rank tangent to the hyperelliptic locus. Consequently the period mapping is a local embedding away from the hyperelliptic locus and also a local embedding when restricted to the hyperelliptic locus.

The involution $u\mapsto -u$ of $H_1$ induces an isomorphism
$$
(A;\ba_1,\dots,\ba_g,\bb_1,\dots,\bb_g)  \to (A;-\ba_1,\dots,-\ba_g,-\bb_1,\dots,-\bb_g).
$$
Since a genus $g$ curve $C$ is hyperelliptic if and only if there is an isomorphism
$$
(C;\ba_1,\dots,\ba_g,\bb_1,\dots,\bb_g)  \to (C;-\ba_1,\dots,-\ba_g,-\bb_1,\dots,-\bb_g)
$$
the period mapping $\pi : \cT_g \to \h_g$ is an \'etale 2:1 covering onto its image away from the hyperelliptic locus where it is an embedding.

\section{The Green--Griffiths invariant of the Ceresa cycle}

Suppose that $C$ is a smooth projective curve of genus $g\ge 3$. Set
$$
A = \Gr^{-1}_F H_1(C) \text{ and } B = \Gr^0_F H_1(C) \cong H^0(\Omega^1_C).
$$
The polarization induces an isomorphism $A\cong B^\vee$. The cotangent space $T_{[\Jac C]}^\vee\cA_g$ to $\cA_g$ at the moduli point $[\Jac C]$ of $\Jac C$ is isomorphic to $S^2 B$.

The fiber $V$ of $\bV$ over $C$ is the primitive {\em quotient}
$$
PH_3(\Jac C)(-1) := H_3(\Jac C)(-1)/\theta\cdot H_1(C)
$$
of $H_3(\Jac C)(-1)$. The canonical polarization $PH^3(\Jac C) \otimes PH^3(\Jac C) \to \Q$ induces an isomorphism
$$
PH_3(\Jac C) \to PH^3(\Jac C)(2).
$$
We therefore have isomorphisms
$$
\begin{tikzpicture}
\matrix[matrix of math nodes,row sep=1mm,column sep=0mm]
{
\Gr_F^{-2} V & \cong  & V^{-2,1} & = & PH^{0,3}(\Jac C) & \cong & \Lambda^3 A \\
\Gr_F^{-1} V & \cong  & V^{-1,0} & = & PH^{1,2}(\Jac C) & \cong & \displaystyle{\frac{\Lambda^2 A \otimes B}{\theta\cdot A}}\\
\Gr_F^{0} V & \cong & V^{0,-1} & = & PH^{2,1}(\Jac C) & \cong & \displaystyle{\frac{ A \otimes \Lambda^2 B}{\theta\cdot B}} \\
\Gr_F^{1} V & \cong & V^{1,-2} & = & PH^{3,0}(\Jac C) & \cong & \Lambda^3 B \\
};
\end{tikzpicture}
$$
where the polarization $\theta$ corresponds to the identity in $\End(B) \cong A\otimes B$.

The Green--Griffiths invariant $\deltabar_C(\nu)$ of $\nu$ at $C$ lies in the homology of the complex $\Gr_F^0(V\otimes \Lambda^\bdot T_{[C]}^\vee\cM_g)$:
\begin{equation}
\label{eqn:ceresa_cplex}
\frac{ A \otimes \Lambda^2 B}{\theta\cdot B} \to \frac{\Lambda^2 A \otimes B}{\theta\cdot A}  \otimes H^0(\Omega^{\otimes 2}_C) \to \Lambda^3 A \otimes \Lambda^2 H^0(\Omega_C^{\otimes 2}). 
\end{equation}
It is the pullback along the period mapping $\cM_g \to \cA_g$ of the fiber
\begin{equation}
\label{eqn:jacobian_cplex}
\frac{ A \otimes \Lambda^2 B}{\theta\cdot B} \to \frac{\Lambda^2 A \otimes B}{\theta\cdot A}  \otimes S^2 B \to \Lambda^3 A \otimes \Lambda^2 S^2 B
\end{equation}
over $[\Jac C] \in \cA_g$ of the complex of vector bundles $\Gr_F^0 (\cV \otimes \Omega^\bdot_{\cA_g})$. The differential is induced by the map
$$
\nabla : B \to A \otimes S^2 B
$$
that is adjoint to the projection $B^{\otimes 2} \to S^2 B$. It is a complex in the category of $\GL(B)$ modules. As such, it does not depend on the point $[\Jac C]$ of $\cA_g$.

The derivative $\nabla \nu_\R$ of $\nu$ at $[C]$ can be regarded as the complex linear map
$$
\nablabar_C \nu : T_{[C]} \cM_g \to \Gr_F^{-1} V = \frac{\Lambda^2 A \otimes B}{\theta\cdot A}
$$
via the prescription in Proposition~\ref{prop:derivatives}. The corresponding point in
$$
\frac{\Lambda^2 A \otimes B}{\theta\cdot A} \otimes H^0(\Omega^{\otimes 2}_C)
$$
is a 1-cocycle in the complex (\ref{eqn:ceresa_cplex}) and represents the Green-Griffiths invariant $\deltabar_C(\nu)$ of $\nu$ at $[C]$ in its homology.

\begin{remark}
The complex (\ref{eqn:ceresa_cplex}) and the Griffiths invariant of the Ceresa cycle have been studied by various authors. See \cite{collino-pirola} and the references therein.
\end{remark}

\section{Symmetry}

Symmetry can be a useful tool for computing $\nablabar_C \nu$. The isotropy group $G_C$ of a point of $\cT_g$ corresponding to a framing of $C$ is a subgroup of $\Sp_g(\Z)$ isomorphic to $\Aut C$. 

Set $\bV = \Lambda^3_0 \bH$, regarded as a variation of Hodge structure over $\cT_g$ and $\cV = \bV \otimes \cO_{\cT_g}$.

\begin{proposition}
\label{prop:symmetry}
Suppose that $g\ge 3$. If $C$ is a smooth projective curve of genus g, then $\Aut C$ acts on the stalk at $[C]$ of the complex $\cV\otimes \Omega_{\cM_g}^\bdot$ and preserves its Hodge filtration. It fixes
\begin{enumerate}

\item the Griffiths invariant $\delta_C(\nu)$ of the normal function of the Ceresa cycle in the stalk at $[C]$ of $\cH^1(F^0 \cV\otimes \Omega_{\cM_g}^1)$,

\item the cocycle $\nablabar_C \nu \in \Gr_F^{-1} V_C \otimes T_C^\vee \cM_g$,

\item the Green--Griffiths invariant $\deltabar_C(\nu)$ of $\nu$ at $C$,

\item the derivative $\nabla \nu_\R$ at $[C]$ of the real lift of $\nu$.

\end{enumerate}
\end{proposition}

\begin{proof}
The group $\Sp_g(\Z)$ acts on the diagram
$$
\begin{tikzcd}
J(\bV) \ar[r] & \cT_g \ar[l,bend right,"\nu"]
\end{tikzcd}
$$
It follows that $\Aut C$ acts on the restriction of this family to an analytic neighbourhood of $[C]$. Now apply Lemma~\ref{lem:symmetry}.
\end{proof}

\part{The Ceresa cycle in genus 3}
\label{part:genus3}

In this part we further investigate the Green--Griffiths invariant of the Ceresa normal function in genus 3 and use a result of Collino and Pirola \cite{collino-pirola} to prove the base case of Theorem~\ref{thm:max-rk}. We then give a new proof of the Collino--Pirola Theorem using Teichm\"uller forms and use it to prove Theorem~\ref{thm:gg-genus3}.

\section{The Green--Griffiths invariant in genus 3}
\label{sec:genus3proof}

In genus 3, the period map $\cM_3 \to \cA_3$ is a local biholomorphism away from the hyperelliptic locus. It induces an isomorphism of the complex (\ref{eqn:ceresa_cplex}) with (\ref{eqn:jacobian_cplex}), which enables us, when working at a non-hyperelliptic curve, to work locally on $\cA_3$ and exploit the action of $\GL(B)$ on the complex (\ref{eqn:jacobian_cplex}).

\subsection{Some representation theory}

Denote the top exterior power of a representation $V$ by $\det V$. Since the pairing $B \otimes \Lambda^2 B \to \det B$ is non-singular in genus 3, we have a canonical isomorphism
$$
\Lambda^2 B \cong A \otimes \det B
$$
of $\GL(B)$ modules. Similarly, $\Lambda^2 A \cong B \otimes \det A$. We also have a canonical isomorphism
$$
S^2 S^2 B \cong S^4 B \oplus S^2 \Lambda^2 B
$$
which holds for all $g \ge 2$. When $g=3$, we thus have a canonical isomorphism
$$
S^2 S^2 B \cong S^4 B \oplus (S^2 A \otimes (\det B)^2).
$$
Both summands are irreducible.

\begin{lemma}
\label{lem:rep}
In genus 3, the are natural $\GL(B)$ module isomorphisms
$$
\frac{\Lambda^2 A \otimes B}{\theta \cdot A} \cong S^2 B \otimes \det A \text{ and } 
\frac{A \otimes \Lambda^2 B}{\theta \cdot B} \cong S^2 A \otimes \det B.
$$
\end{lemma}

\begin{proof}
The first isomorphism follows from the isomorphisms
$$
\Lambda^2 A \otimes B \cong B^{\otimes 2} \otimes \det A \cong (S^2 B \oplus \Lambda^2 B) \otimes \det A \cong (S^2 B\otimes \det A) \oplus A.
$$
The second is proved similarly, or by taking duals.
\end{proof}

\subsubsection{Green--Griffiths cohomology}

The following result combines computations of Nori \cite[pp.~371--372]{nori} and Collino--Pirola \cite[Lem.~4.2.3]{collino-pirola}.

\begin{proposition}
\label{prop:coho}
If $C$ is a non-hyperelliptic curve of genus 3, then 
\begin{enumerate}

\item If $p \ge 0$, then $H^0(\Gr_F^p(V\otimes \Lambda^\bdot S^2 B))$ vanishes.

\item If $p > 0$, then $H^1(\Gr_F^p(V\otimes \Lambda^\bdot S^2 B))$ vanishes.

\item\label{item:cocycles} The space of 1-cocycles in the complex $\Gr_F^0 (V\otimes \Lambda^\bdot S^2 B)$ is naturally isomorphic to
$$
S^2 S^2 B \otimes \det A \cong (S^4 B \otimes \det A) \oplus (S^2 A \otimes \det B)
$$
and the space of 1-coboundaries to $S^2 A \otimes \det B$. Consequently,
$$
H^1(\Gr_F^0 (V\otimes \Lambda^\bdot S^2 B)) \cong S^4 B \otimes \det A.
$$
\end{enumerate}
\end{proposition}

\begin{proof}
The complexes $\Gr_F^p (V \otimes \Lambda^\bdot S^2 B)$ are:
$$
\begin{tikzcd}[row sep=1mm]
\Gr_F^2: & 0 \ar[r] & \Lambda^3 B \otimes S^2 B \ar[r] & \displaystyle{\frac{ A \otimes \Lambda^2 B}{\theta\cdot B}}\otimes \Lambda^2 S^2 B
\\
\Gr_F^1: & \Lambda^3 B \ar[r] & \displaystyle{\frac{ A \otimes \Lambda^2 B}{\theta\cdot B}}\otimes S^2 B \ar[r] & \displaystyle{\frac{\Lambda^2 A \otimes B}{\theta\cdot A}}  \otimes \Lambda^2 S^2 B
\\
\Gr_F^0: & \displaystyle{\frac{ A \otimes \Lambda^2 B}{\theta\cdot B}} \ar[r] & \displaystyle{\frac{\Lambda^2 A \otimes B}{\theta\cdot A}}  \otimes S^2 B \ar[r] &  \Lambda^3 A \otimes \Lambda^2 S^2 B
\end{tikzcd}
$$
Each is a complex of $\GL(B)$ modules. The vanishing of the homology of the first two rows in degrees 0 and 1 is sketched by Nori \cite[\S7]{nori}. Here we sketch an elementary proof. The exactness of the first row $\Gr_F^2$ follows from the fact that $S^2 B\otimes \det B$ is irreducible and the differential is non-zero and therefore injective. Exactness of the second row $\Gr_F^1$ follows (using Lemma~\ref{lem:rep}) from the fact that
$$
\displaystyle{\frac{ A \otimes \Lambda^2 B}{\theta\cdot B}}\otimes S^2 B \cong \End (S^2 B) \otimes \det B
$$
which has 3 irreducible factors. The differential takes $\Lambda^3 B$ to $\det B\otimes \id_{S^2B}$. The other two factors map injectively into the degree 2 term. So we focus on the complex $\Gr_F^0$.

Lemma~\ref{lem:rep} implies that the degree 0 term is $S^2 A \otimes \det B$ and that the degree 1 term is
\begin{align*}
S^2 B \otimes S^2 B \otimes \det A &\cong (S^2 S^2 B\otimes \det A) \oplus (\Lambda^2 S^2 B \otimes \det A)
\cr
&\cong (S^4 B \otimes \det A) \oplus (S^2 A \otimes \det B) \oplus (\Lambda^2 S^2 B \otimes \det A).
\end{align*}
Since the degree 0 term is irreducible and the differential is non-zero, we see that the differential is injective and that the group of 1-coboundaries is $S^2 A\otimes \det B$. Since the degree 2 term is irreducible and  not isomorphic to either of the first 2 components of the degree 1 term, and since second differential is non-zero, it follows that the group of 1-cocycles is $(S^4 B\otimes \det A) \oplus (S^2 A \otimes \det B)$ and that $H^1(\Gr_F^0)$ is $S^4 B \otimes \det A$.
\end{proof}

Since each $\Gr_F^p(\cV \otimes \Omega_{\cM_3}^\bdot)$ is a complex of vector bundles over the complement of the hyperelliptic locus, the previous result implies:

\begin{corollary}
When $g=3$, the differential $\nabla : F^0 \cV \to \cV \otimes \Omega^1_{\cM_3}$ is injective on the complement of the hyperelliptic locus. In addition, the homology sheaves $\cH^1(\Gr^p_F(\cV\otimes \Omega_{\cM_3}^\bdot))$ vanish on the complement of the hyperelliptic locus when $p>0$.
\end{corollary}

Since $\P(H^0(\Omega^1_C))^\vee = \P(A)$, the Green--Griffiths invariant $\deltabar_C(\nu)$, if it is non-zero, defines a plane quartic in $\P(A)$.

\begin{theorem}[Collino--Pirola {\cite[Thm.~4.2.4]{collino-pirola}}]
\label{thm:collino-pirola}
If $C$ is a non-hyperelliptic curve of genus 3, then $\deltabar_C(\nu)$ is a non-zero quartic polynomial that defines the canonical image of $C$ in $\P(A)$.
\end{theorem}

We will give a new proof of it in Section~\ref{sec:proof-cp} using Teichm\"uller modular forms.

\subsection{Linear algebra}
\label{sec:lin-alg}

Suppose that $C$ is not hyperelliptic of genus 3. We will regard $S^2 S^2 B \otimes \det A$ as the space of symmetric bilinear forms $S^2 A \otimes S^2 A \to \det A$. In view of Proposition~\ref{prop:coho}(\ref{item:cocycles}), we can thus regard the derivative $\nablabar_C \nu$ as a symmetric bilinear form
$$
D_C : S^2 A \otimes S^2 A \to \det A.
$$
It decomposes as the sum
$$
D_C = Q_C + R_C
$$
of two symmetric forms, where $Q_C \in S^4 B \otimes \det A$, a natural representative of $\deltabar_C(\nu)$, and $R_C \in S^2 A \otimes \det B$, which is a coboundary.

Fix a volume form $\vol_A \in \det A$ and let $\vol_B \in \det B$ be the dual volume form. Identify $\det A$ and $\det B$ with $\C$ via the isomorphisms
$$
\C \to \det A,\quad t \mapsto t\, \vol_A \text{ and } \C \to \det B, \quad t \mapsto t\, \vol_B.
$$
The pairing $Q_C : S^2 A \otimes S^2 A \to \C$ is (up to a non-zero multiple) given by the formula
$$
Q_C(u,v) = f(uv)/24
$$
where $f \in S^4 B$ is a generator of the ideal of functions that vanish on the canonical image of $C$ in $\P(A)$ and where $uv\in S^4 A$ is the product of $u,v\in S^2 A$.

The pairing $R_C$ is the composite
$$
S^2 A \to S^2 \Lambda^2 B \to (S^2 B)^{\otimes 2} \to \Hom(S^2 A \otimes S^2 A,\C)
$$
where the first map is induced by the isomorphism
$$
A \to \Lambda^2 B,\quad a \mapsto a \intprod \vol
$$
and the second by
$$
(b_1\wedge b_2) \cdot (b_1'\wedge b_2') \mapsto
b_1b_1'\otimes b_2b_2' + b_2b_2'\otimes b_1b_1' - b_1b_2' \otimes b_2b_1' - b_2b_1' \otimes b_1b_2'.
$$

We now give explicit formulas (up to multiplication by a non-zero scalar) of them. Fix a basis $\ee_0,\ee_1,\ee_2$ of $A$ such that $\vol_A = \ee_0 \wedge \ee_1 \wedge \ee_2$. Let $x_0,x_1,x_2$ be the dual basis of $B$.

\begin{proposition}
\label{prop:QC}
If the defining equation $f$ of $C$ in $\P(A)$ is
$$
f(x_0,x_1,x_2) = \sum_j a_j x_j^4 + 4\sum_{j\neq k} b_{jk} x_j x_k^3 + 6\sum_{j<k} c_{jk} x_j^2 x_k^2 + 12\sum_j\sum_{k,\ell\neq j} d_j x_j^2x_k x_\ell,
$$
then the matrix of $Q_C$ with respect to the basis
$$
\ee_0^2,\ \ee_1^2,\ \ee_2^2,\ \ee_0\ee_1,\ \ee_0\ee_2,\ \ee_1 \ee_2
$$
of $S^2 A$ is a non-zero multiple of
\begin{equation}
\label{eqn:GC}
\begin{pmatrix}
a_0 & c_{01} & c_{02}  & b_{10}  & b_{20}  & d_0  
\\
c_{01}  & a_1  & c_{12}  & b_{01}  & d_1  & b_{21}  
\\
c_{02}  & c_{12}  & a_2  & d_2  & b_{02}  & b_{12}  
\\
b_{10}  & b_{01}  & d_2  & c_{01}  & d_0  & d_1  
\\
b_{20}  & d_1  & b_{02}  & d_0  & c_{02}  & d_2  
\\
d_0  & b_{21}  & b_{12}  & d_1  & d_2  & c_{12}  
\end{pmatrix}
\end{equation}
\end{proposition}

\begin{proof}[Sketch of proof]
The $\GL(A)$-invariant dual paring $S^4 B \otimes S^4 A \to \C$ is unique up to a constant. We choose the normalization
$$
\langle x_{j_1}x_{j_2}x_{j_3}x_{j_4}, \ee_{k_1}\ee_{k_2}\ee_{k_3}\ee_{k_4} \rangle
= \sum_{\sigma \in \bS_4} \prod_{i=1}^4 \langle x_{j_i}, \ee_{k_{\sigma(i)}} \rangle,
$$
where $\bS_4$ is the symmetric group on 4 letters. So, for example,
$$
f(\ee_0^4) = 24\, a_0,\ f(\ee_0^3 \ee_1) = 4\cdot 6\, b_{10},\ f(\ee_0^2\ee_1^2) = 6 \cdot 4\, c_{01},\ f(\ee_0^2\ee_1\ee_2) = 12\cdot 2\, d_0.
$$
Thus
\begin{multline*}
Q_C(\ee_0^2,\ee_0^2) = f(\ee_0^4)/24 = a_0,\ Q_C(\ee_0^2,\ee_0\ee_1) = f(\ee_0^3 \ee_1)/24 = b_{10}, \cr
Q_C(\ee_0^2,\ee_1^2) = Q_C(\ee_0\ee_1,\ee_0\ee_1) = f(\ee_0^2\ee_1^2)/24 = c_{01}, \cr
Q_C(\ee_0^2,\ee_1\ee_2) = Q_C(\ee_0\ee_1,\ee_0\ee_2) = f(\ee_0^2\ee_1\ee_2)/24 = d_0. 
\end{multline*}
The remaining entries are obtained by permuting the indices.
\end{proof}

\begin{remark}
For the sake of completeness, we give a formula for the bilinear form $R_C$. If
$$
h = \sum_j p_j \ee_j^2 + 2\sum_{j<k} q_{jk} \ee_j \ee_k \in S^2 A
$$
is the projection of $\nablabar_C \nu$ onto $S^2 A\otimes \det B$, then the matrix of $R_C$ with respect to the basis
$$
\ee_0^2,\ \ee_1^2,\ \ee_2^2,\ \ee_0\ee_1,\ \ee_0\ee_2,\ \ee_1 \ee_2
$$
of $S^2 A$ is a non-zero multiple of
$$
\begin{pmatrix}
0 & p_2 & p_1 & 0 & 0 & -q_{12}
\\
p_2 & 0 & p_0 & 0 & -q_{02} & 0
\\
p_1 & p_0 & 0 & -q_{01} & 0 & 0
\\
0 & 0 & -q_{01} & -p_2 & q_{12} & q_{02}
\\
0 & -q_{02} & 0 & q_{12} & -p_1 & q_{01}
\\
-q_{12} & 0 & 0 & q_{02} & q_{01} & -p_0
\end{pmatrix}
$$

To prove this, one easily checks that the map $S^2 A \to (S^2B)^{\otimes 2}$ satisfies
\begin{align*}
\ee_0^2 &\mapsto x_1^2 \otimes x_2^2 + x_2^2 \otimes x_1^2 - 2 x_1x_2 \otimes x_1 x_2 \cr
\ee_0\ee_1 &\mapsto  x_0 x_2 \otimes x_1 x_2 + x_1 x_2 \otimes x_0 x_2 - x_0 x_1 \otimes x_2^2 - x_2^2 \otimes x_0 x_1.
\end{align*}
The formulas for the images of the other basis vectors is obtained by cyclically permuting the indices.
\end{remark}

\subsection{The Klein quartic}

Here we prove Theorem~\ref{thm:max-rk} in genus 3 by checking that $\nablabar\nu$ has maximal rank at the moduli point of the Klein quartic $C$. We do this by using the symmetries of $C$ to show that the tensor $R_C$ vanishes. 

The Klein curve is the projective completion $X(7)$ of the modular curve $Y(7) = \G(7)\bs \h$, where $\G(7)$ is the full level 7 subgroup of $\SL_2(\Z)$. It has genus 3 and automorphism group $\PSL_2(\bF_7)$, a simple group of order 168. Its canonical image is the plane quartic
$$
x^3y+y^3z+z^3x=0.
$$
The corresponding matrix (\ref{eqn:GC}) has rank 6 and determinant $-729 = -3^6$. To show that $\nablabar_C(\nu)$ has maximal rank it suffices to show that $R_C$ vanishes.

\begin{proposition}
If $C$ is the Klein quartic, then $R_C = 0$. Consequently, $\nablabar \nu$ has maximal rank at $[C] \in \cM_3$.
\end{proposition}

\begin{proof}
Set $G=\PSL_2(\bF_7)$. In view of Proposition~\ref{prop:symmetry}, it suffices to show that $S^2 A\otimes \det B$ does not contain the trivial representation of $G$. The irreducible complex representations of $G$ have dimensions 1, 3, 3, 6, 7, 8. The two 3-dimensional representations are conjugate. Since $\SL_2(\Z)$ has no cusp forms of weight 2, $B = H^0(\Omega^1_C)$ does not contain the trivial representation. It is thus one of the two 3-dimensional irreducible representations and $A$ is its conjugate. Since the only 1-dimensional $G$-module is the trivial representation,  $\det A$ and $\det B$ are trivial.

Both $S^2 A$ and $S^2 B$ are isomorphic to the unique 6-dimensional irreducible representation. It follows that $S^2 A\otimes \det B$ is irreducible and thus contains no copy of the trivial representation.
\end{proof}

\begin{remark}
One can easily check that $S^4 B$ contains one copy of the trivial representation. It is spanned by a defining equation of $C$ in $\P(A)$.
\end{remark}

Since $\nablabar \nu$ is a holomorphic section of $\cV\otimes \Omega_{\cM_3}^1$, we conclude that $\nablabar \nu$ has maximal rank on an open dense subset of $\cM_3$.

\begin{proposition}
The derivative $\nablabar \nu$ has maximal rank on a dense open subset of $\cM_3$.
\end{proposition}

\section{Geometry of the period map in genus 3}

Here we supplement the discussion of the geometry of the period map in Section~\ref{sec:period-map} by giving a more detailed description of it in genus 3, which is needed in our proof of Theorem~\ref{thm:collino-pirola}. This enables us to show that the Green--Griffiths invariant $\deltabar(\nu)$ is a Teichm\"uller modular form and to show that the derivative of $\nu$ along the hyperelliptic locus is the restriction of the Siegel modular form $\chi_{4,0,8}$ to the hyperelliptic locus.

Denote the closure of the locus in $\h_3$ of jacobians of hyperelliptic curves by $\h_3^\hyp$ and the locus of reducible (as a principally polarized abelian variety) abelian 3-folds by $\h_3^\red$. Since every genus 2 curve is hyperelliptic, every reducible abelian 3-fold is the jacobian of a hyperelliptic curve of compact type. So we have a filtration
$$
\h_3 \supset \h_3^\hyp \supset \h_3^\red.
$$
The image of the genus 3 period mapping
 $\cT_3 \to \h_3$ is $\h_3 - \h_3^\red$.
(This is surely well known. For a proof see \cite[Cor.~7]{hain:A3}.) Each even theta characteristic $\alpha$ determines a theta null $\vartheta_\alpha : \h_3 \to \C$. Their product is the Siegel modular form
$$
\chi_{18} := \prod_{\alpha \text{ even}} \vartheta_\alpha
$$
of weight 18.

\begin{proposition}
The genus 3 Torelli space $\cT_3$ is the analytic subvariety of $(\h_3-\h_3^\red) \times \C$ defined by $y^2 = \chi_{18}(\tau)$.
\end{proposition}

\begin{proof}
Since a smooth genus 3 curve is hyperelliptic if and only if it has a (necessarily unique) vanishing even theta characteristic\footnote{A theta characteristic $\alpha$ is a square root of the canonical divisor class. It is even if $h^0(\alpha)$ is even. In genus 3, an even theta characteristic is effective if and only if $h^0(\alpha)=2$.}, it follows that $\h_3^\hyp$ is the zero locus of $\chi_{18}$. The result follows as $\cT_3 \to \h_3-\h_3^\red$ is the unique double covering of $\h_3-\h_3^\red$ branched along the smooth divisor $\h_3^\hyp-\h_3^\red$.
\end{proof}

The function $y \in \cO(\cT_3)$ is thus a square root of $\chi_{18}$. In the terminology of Ichikawa \cite{ichikawa} it is a {\em Teichm\"uller form} of genus 3 and is, in fact, the most basic of them. It has weight 9, meaning it is a section of the 9th power of the determinant of the Hodge bundle, and vanishes to order 1 along the hyperelliptic locus $\cT_3^\hyp$ of Torelli space. It has been (re)christened $\chi_9$ by Cl\'ery--Faber--van der Geer in \cite{cfg}.

Denote the coordinates in $\h_3$ by
$$
\tau = \begin{pmatrix}
\tau_{11} & \tau_{12} & \tau_{13} \cr \tau_{12} & \tau_{22} & \tau_{23} \cr \tau_{13} & \tau_{23} & \tau_{33}
\end{pmatrix}.
$$

The genus 3 analogue of the $q$-disk is the quotient of $\h_3$ by the unipotent subgroup $\Z^3$ of $\Sp_3(\Z)$ that acts via
$$
(n_1,n_2,n_3) : \tau \mapsto \tau + \diag(n_1,n_2,n_3).
$$
Set $q_j = \exp(2\pi i \tau_{jj})$. Denote the punctured unit disk by $\bD'$. The map
$$
\Z^3 \bs \h_3 \to (\bD')^3 \times \C^3, \quad \tau \mapsto (q_1,q_2,q_3,\tau_{12},\tau_{13},\tau_{23})
$$
is an open immersion. Define a partial compactification $(\Z^3\bs\h_3)^c$ of $\Z^3\bs \h_3$ by adding the three coordinate hyperplanes $q_j=0$, but not their intersections $q_j=q_k=0$, $j\neq k$. Set
$$
U = (\Z^3\bs\h_3)^c - \text{closure of }(\Z^3\bs \h_3)^\red.
$$
Let $D_j$ be the divisor $q_j=0$ and $D$ be their union.

The Siegel modular form $\chi_{18}$ descends to a function $U \to \C$, which we will also denote by $\chi_{18}$. We have
\begin{equation}
\label{eqn:div-chi18}
\div(\chi_{18}) = H + 2D
\end{equation}
where $H$ denotes the class of the image of $\h_3^\hyp$. (See \cite[\S3]{cfg}.) Denote the normalization of the zero locus of
$$
y^2 - \chi_{18} : U \times \C \to \C
$$
by $V$. It is a smooth double covering of $U$ branched over the hyperelliptic locus $U^\hyp$. Since $\chi_{18}$ vanishes to order 2 along $D$, the projection $\pi : V \to U$ is not ramified over $D$, only over $U^\hyp$. Denote the inverse image of $D$ in $V$ by $\Delta$. It lies over the smooth locus of the boundary divisor $\Delta_0$ of $\cMbar_3$. The divisor $D$ lies over a boundary divisor of each toroidal compactification of $\cA_3$.

The modular form $\chi_9$ extends to a holomorphic function $V \to \C$, also denoted $\chi_9$. We have
\begin{equation}
\label{eqn:div-chi9}
\div(\chi_9) = H + \Delta.
\end{equation}

The representations $A$ and $B$ of $\GL(B)$ correspond to automorphic vector bundles $\sA$ and $\sB$ over $\cA_3$, where $\sB$ is the Hodge bundle. Their pullbacks to $\Z^3\bs \h_3$ and its double covering (also denoted $\sA$ and $\sB$) extend to $U$ (resp.\ $V$) as graded quotients of the Hodge filtration of Deligne's canonical extension \cite{deligne:diff_eq} of $\bH$ to $U$ (resp.\ $V$). When $\sA$ is regarded as a sheaf over $X$, where $X=\cA_3, \h_3, \cT_3, V,\dots$, we will denote it by $\sA_X$. Similarly with $\sB$.

The Teichm\"uller modular form $\chi_9$ is a section of $(\det \sB)^9$ over $\cM_3$, $\cT_3$ and $V$. It vanishes to order 1 on the hyperelliptic locus $V^\hyp$. Denote the inclusion $V^\hyp \to V$ by $j$.

\begin{lemma}
\label{lem:Omega}
The vector bundle $\pi^\ast \Omega^1_U(\log D)$ over $V$ is isomorphic to $S^2\sB_V$. Consequently, we have a short exact sequence
$$
0 \to \pi^\ast \Omega^1_U(\log D) \to \Omega^1_V(\log \Delta) \to j^\ast (\det\sB_V)^9 \to 0.
$$
\end{lemma}

\begin{proof}
Since $S^2\sB_U$ is isomorphic to $\Omega_U^1(\log D)$, \cite[p.~117]{faltings-chai}, it follows that $S^2\sB_V$ is isomorphic to $\Omega^1_V(\log \Delta)$ away from the hyperelliptic locus. Since the hyperelliptic locus in $V$ has reduced defining equation $\chi_9=0$, it has conormal bundle $j^\ast (\det\sB_V)^9$. The exact sequence follows as $\pi : V \to U$ is ramified along $V^\hyp$.
\end{proof}

\section{Proof of Theorem~\ref{thm:gg-genus3}}
\label{sec:proof-cp}

Theorem~\ref{thm:gg-genus3} is a consequence of Theorem~\ref{thm:gg-genus3-final} below. Set $\cT = \cT_3$, $\h = \h_3$ and $\h' = \h_3 - \h_3^\red$. We will work over $\cT$, $\h'$, $U$ and $V$, as appropriate. Their hyperelliptic loci will be denoted by $\cT^\hyp$, $\h^\hyp$, $V^\hyp$, etc.

\begin{lemma}
\label{lem:technical}
\phantom{xx}
\begin{enumerate}[(a)]

\item
\label{item:real-anal}
The tensor $\nablabar \nu$ is a real analytic section of the holomorphic vector bundle $\Gr_F^{-1}\cV \otimes \Omega^1_\cT$ over $\cT$.

\item
\label{item:gg-log}
The Green--Griffiths invariant $\deltabar(\nu)$ is a holomorphic section of $\Gr_F^{-1}\cV\otimes \Omega_V^1(\log \Delta)$ over $V$.

\item
\label{item:decomp}
The kernel of $\nablabar : \Gr_F^{-1} \cV \otimes \Omega^1_V(\log\Delta) \to \Gr_F^{-2} \cV \otimes \Omega_V^2(\log\Delta)$ has a natural direct sum decomposition
$$
\nablabar \Gr_F^0 \cV_V \oplus (S^4\sB_V \otimes \det\sA_V)^\sim
$$
where the space of sections of $(S^4\sB_V \otimes \det\sA_V)^\sim$ over the open subset $O$ of $V$ is
$$
\{\omega \in H^0(O,\Gr_F^{-1} \cV \otimes \Omega_V^1(\log\Delta)) : \chi_9\omega \in H^0(O,S^4 \sB_V \otimes \det\sA_V)\}.
$$

\item
\label{item:ses}
There is an exact sequence
$$
0 \to S^4 \sB_V \otimes \det \sA_V \to (S^4 \sB_V \otimes \det \sA_V)^\sim \to j_\ast \big[S^4 \sB_V \otimes (\det\sB_V)^8\big]
$$
where $j : V^\hyp \to V$ is the closed inclusion of the hyperelliptic locus.

\item
\label{item:decomp-nu}
There is a unique real analytic section $f$ of $\Gr_F^0 \cV_V$ over $V-\Delta$ and a unique holomorphic section $\deltatilde(\nu)$ of $(S^4\sB_V \otimes \det\sA_V)^\sim$ over $V$ such that
$$
\nablabar \nu = \deltatilde(\nu) + \nablabar f.
$$
The section $\deltatilde(\nu)$ is anti-invariant under the involution. Both $f$ and $\nablabar f$ vanish on the hyperelliptic locus.

\end{enumerate}
\end{lemma}

\begin{proof}
At two points in the proof, we will appeal to general results about normal functions from Part~\ref{part:vmhs}. These do not depend on results in the preceding sections. Denote the sheaf of real analytic $j$ forms on a variety $X$ by $\sE_X^j$ and the sheaf of real analytic $(p,q)$ forms by $\sE_X^{p,q}$. Set $\sE_X = \sE_X^0$.

The map $\nablabar : \Gr_F^0 \cV \to \Gr_F^{-1}\cV \otimes \Omega^1_{\cT}$ of holomorphic vector bundles extends to an $\sE_{\cT}$-linear map
$$
\nablabar' : \Gr_F^0 \cV\otimes_{\cO_\cT} \sE_{\cT} \to \Gr_F^{-1}\cV \otimes_{\cO_\cT} \sE_{\cT}^{1,0},
$$
where $\nablabar'$ denotes the map induced by the $(1,0)$ component of
$$
\nabla : \cV\otimes_{\cO_\cT} \sE_{\cT} \to \cV\otimes_{\cO_\cT} \sE_{\cT}^1.
$$
By Corollary~\ref{cor:nablabar}, $\nablabar \nu$ is the image of $\nablabar'\nu_\R$ in $\Gr_F^{-1}\cV \otimes_{\cO_\cT} \sE_\cT^{1,0}$. It is real analytic. This proves (\ref{item:real-anal}). Proposition~\ref{prop:anf} implies that $\nu$ is admissible. So (\ref{item:gg-log}) follows from Corollary~\ref{cor:gg-log}. 

Proposition~\ref{prop:coho}(\ref{item:cocycles}) implies that the kernel of $\nablabar : \Gr_F^0 \cV_U \to \Gr_F^{-1}\cV \otimes \Omega^1_U(\log D)$ is isomorphic to $(S^4\sB_U \otimes \det \sA_U) \oplus (S^2 \sA_U \otimes \det \sB_U)$. Assertion (\ref{item:decomp}) follows from this and the fact that, since $\nablabar$ is $\cO_V$-linear, $\nablabar \Gr_F^0 \cV_V$ is a subsheaf of $\pi^\ast \Gr_F^{-1}\cV\otimes \Omega^1_U(\log D)$.

Denote the hyperelliptic locus of $V$ by $H$. The exactness of the sequence in (\ref{item:ses}) follows from the fact that $(S^4\sB_V\otimes \det\sA_V)^\sim$ is a subsheaf of $S^4\sB_V\otimes \det\sA_V(H)$ and that the conormal bundle of $H$ in $V$ is the restriction of $(\det \sB_V)^9$ to $H$. This is because $H$ is defined by $\chi_9$, which vanishes to order 1 on $H$ and is a section of $(\det\sB)^9$.

To prove (\ref{item:decomp-nu}) first choose a local holomorphic lift $\nutilde$ of $\nu$ defined on an open set $O$ of $V$. We can write
$$
\nablabar \nutilde = (\nablabar h, \delta(\nutilde)) \in H^0(O,\nablabar\Gr_F^0 \cV) \oplus H^0(O,(S^4\sB\otimes \det\sA)^\sim),
$$
where $h$ is a section of $\Gr_F^0 \cV_V$. Both components are holomorphic on $V$. Since two lifts of $\nu$ differ by a section of $F^0\cV_V$, the second component does not depend on the choice of the lift $\nutilde$. Call it $\deltatilde(\nu)$. The existence of $f$ follows from the fact that $\nablabar : \Gr_F^0 \cV\otimes_{\cO_\cT} \cE_\cT \to \Gr_F^{-1}\cV\otimes_{\cO_\cT} \cE^{1,0}$ is injective, so there is a unique real analytic section $f$ of $\Gr_F^0 \cV$ over $V-\Delta$ with
$$
\nablabar \nu := \nablabar' \nu_\R = (\nablabar f, \deltatilde(\nu)).
$$
The vanishing of $f$ and $\nablabar f$ on the hyperelliptic locus follows as $\nablabar\nu$ is anti-invariant under the involution of $\sigma : V\to U$, which forces $f$ to vanish on the hyperelliptic locus, and as $\nablabar$ is a $\sigma$-invariant map of vector bundles.
\end{proof}

We now assume familiarity with \cite{cfg} and use some of the notation in it, especially the notation for the weight $(j,k,\ell)$ of Siegel and Teichm\"uller modular forms.\footnote{Briefly, fix a maximal torus of $\GL(B)$. A holomorphic (resp.\ meromorphic) modular form of degree (or genus) 3 and weight $(j,k,\ell)$ is a holomorphic (resp.\ meromorphic) section of the automorphic vector bundle over $\h_3$ that corresponds to the irreducible $\GL(B)$ module with a highest weight vector of toral weight $[j+k+\ell,k+\ell,\ell]$. It is invariant under the natural $\Sp_3(\Z)$ action and thus descends to a section defined over $U-D$. Siegel modular forms extend to sections defined over $U$. No growth conditions are imposed along the divisor $D$; they are imposed only in genus 1. We choose the torus so that sections of the vector bundle $\sB$ are modular forms of weight $(1,0,0)$, which implies that sections of $\sA$ have weight $(0,1,-1)$. With this convention $\chi_{18}$ has weight $(0,0,18)$ and $\chi_9$ has weight $(0,0,9)$.}

\begin{theorem}
\label{thm:gg-genus3-final}
The Green--Griffiths invariant $\deltabar(\nu)$ of the Ceresa cycle is a section of the subsheaf $(S^4\sB_V \otimes\det \sA_V)^\sim$ of $S^4\sB_V \otimes\det \sA_V(H)$. It is a non-zero multiple of the Teichm\"uller modular form $\chi_{4,0,-1}$. Its restriction to the hyperelliptic locus of $V$ projects, under the projection in Lemma~\ref{lem:technical}(\ref{item:ses}), to a non-zero multiple of the section $\chi_{4,0,8}$ of $S^4\sB_V\otimes (\det\sB_V)^8$.
\end{theorem}

\begin{proof}
We can identify $\deltabar(\nu)$ with its canonical representative $\deltatilde(\nu)$ as both are sections of $S^4 \sB \otimes \det \sA$ over $V$. Lemma~\ref{lem:technical}(\ref{item:decomp-nu}) implies that $\chi_9 \deltatilde(\nu)$ is a section of $S^4 \sB_V \otimes \det \sA_V$ that is invariant under the involution of $V$. It is therefore a Siegel modular form of weight $(4,0,8)$. Since $\chi_9$ vanishes on $\Delta$, it is a cusp form. Ta\"ibi \cite{taibi} has shown that the space of Siegel cusp forms of weight $(4,0,8)$ is one dimensional and spanned by the cusp form $\chi_{4,0,8}$. It follows that $\deltabar(\nu)$ is a multiple of the meromorphic Teichm\"uller modular form
$$
\chi_{4,0,-1} := \chi_{4,0,8}/\chi_9
$$
and that it is a section of $(S^4 \sB_V \otimes \det \sA_V)^\sim$.

This multiple is non-zero: The derivative $\nabla\nu_\R$ of the Ceresa normal function is a de~Rham representative of the class of $\nu$ in $H^1(\cM_3,\bV)$. This class is non-zero as its restriction to the Torelli group $T_3$ is twice the Johnson homomorphism. (See \cite[Thm.~11.1]{hain:msri}. There are many other proofs of non-triviality, such as Ceresa's original proof \cite{ceresa}.) This implies that the Griffiths invariant $\delta(\nu)$ is non-zero. Propositions~\ref{prop:green} and \ref{prop:coho} then imply that $\deltabar(\nu)$ is also non-zero.

Next we compute the restriction of $\deltatilde(\nu)$ to the hyperelliptic locus. We already know that $\chi_9 \deltatilde(\nu)$ descends to a non-zero multiple of the section $\chi_{4,0,8}$ of the vector bundle $S^4\sB_U\otimes (\det\sB_U)^8$. Since $U^\hyp$ is the divisor defined by $\chi_{18} = 0$ and since the restriction of $\nu$ to $V^\hyp$ is zero, we can write
$$
\chi_9\deltatilde(\nu)|_{U^\hyp} = h d\chi_{18}
$$
after identifying $S^4\sB_U\otimes \det\sA_U$ with $\Gr^{-1}_F\cV\otimes\Omega^1_U(\log D)$. Here $h$ is a section of the restriction of $S^4\sB\otimes (\det\sB_U)^8$ to $U^\hyp$ as the conormal bundle of $U^\hyp$ is $(\det\sB_U)^{18}$, which is trivialized by $d\chi_{18}$. It follows that the restriction of $\deltatilde(\nu)$ to $V^\hyp$ is
$$
\deltatilde(\nu)|_{V^\hyp} = h d\chi_{18}/\chi_9 = 2 h d\chi_9.
$$
Since the conormal bundle of $V^\hyp$ is isomorphic to $(\det \sB_V)^9$, which is trivialized by $d\chi_9$, we see that the restriction of $\deltatilde(\nu)$ to $V^\hyp$ can be identified with a non-zero multiple of the restriction of $\chi_{4,0,8}$ to $V^\hyp$.
\end{proof}

\begin{remark}
The result of Collino and Pirola (Thm.~\ref{thm:collino-pirola}) now follows from \cite[Prop.~10.1]{cfg}.
\end{remark}

\subsection{Proof of Corollary~\ref{cor:harris}}

Since the Ceresa normal function vanishes on the hyperelliptic locus, a smooth divisor, the rank of $\nu$ at each point of the hyperelliptic locus is at most 1. By Lemma~\ref{lem:technical}(\ref{item:decomp-nu}), $\nablabar f$ vanishes on the hyperelliptic locus. So, to prove the result, we need to show that $\deltabar(\nu)$ has no zeros on the hyperelliptic locus. By Theorem~\ref{thm:gg-genus3-final}, the restriction of $\deltabar(\nu)$ to the hyperelliptic locus can be identified with $\chi_{4,0,8}$. So it suffices to show that the restriction of $\chi_{4,0,8}$ to the hyperelliptic locus has no zeros.

In genus 3, the restriction of $\sB$ to the hyperelliptic locus (regarded as a stack) is $S^2\sW$, where $\sW$ is the rank 2 vector bundle whose fiber over the moduli point of a hyperelliptic curve is the linear system associated to its hyperelliptic series. By the representation theory of $\GL_2$, the restriction of $S^4\sB$ to the hyperelliptic locus in genus 3 decomposes
$$
S^8 \sW \oplus S^4 \sW \otimes (\det \sW)^2 \oplus (\det \sW)^4
$$
and $\det \sB$ restricts to $(\det \sW)^3$. So the restriction of $S^4\sB\otimes (\det \sB)^8$ to the hyperelliptic locus projects to
$$
S^8\sW \otimes (\det\sW)^{24} = [S^8\sW \otimes (\det\sW)^{-4}] \otimes (\det \sW)^{28}.
$$
By \cite[\S6]{vdg-kouvidakis}, the projection of the restriction of $\chi_{4,0,8}$ to the hyperelliptic locus onto this factor is the product $f_{8,-2}$ of the universal octic times the discriminant $\fd$ of the binary octic, which is a section of $(\det \sW)^{28}$. Since neither vanish at the moduli point of a smooth hyperelliptic curve, $\chi_{4,0,8}$ has no zeros there. 

This strengthens the genus 3 case of the result \cite[Thm.~6.5]{harris} of Bruno Harris. He proved that the derivative of the normal function of the Ceresa cycle has rank 1 at the general hyperelliptic curve of genus $g\ge 3$.

\part{Admissible normal functions}
\label{part:vmhs}

The inductive proof of the general case requires technical Hodge theory. In particular, it requires an understanding of the asymptotic behaviour of the variations of mixed Hodge structure that correspond to normal functions. 

After recalling the definition and basic properties of admissible variations of mixed Hodge structure, we apply them to describe the boundary behaviour of normal functions. We review the construction of the N\'eron model of a family of intermediate jacobians by Green, Griffiths and Kerr \cite{ggk} in the case where the variation is a nilpotent orbit. This is applied to study the asymptotic behaviour of admissible normal functions in codimension 1. In particular, it is used in the definition of {\em residual normal functions}, which plays a key role in the inductive proof of Theorem~\ref{thm:max-rk}.

\section{Admissible variations of mixed Hodge structure}
\label{sec:avmhs}

This is a terse review of variations of mixed Hodge structure. Basic references include the papers of Steenbrink--Zucker \cite{steenbrink-zucker} and Kashiwara \cite{kashiwara}.

Suppose that $S$ is a smooth quasi-projective variety. We write $S= \Sbar - D$, where $\Sbar$ is a smooth projective variety and $D$ is a normal crossing divisor. An admissible variation of $\Z$-MHS $\bE$ over $S$ consists of the following data:
\begin{enumerate}

\item  A local system $\bE_\Z$ over $S$ of finitely generated abelian groups which, for simplicity, we assume to have unipotent local monodromy.\footnote{This condition is satisfied by all variations of MHS in this paper.}

\item An increasing filtration (the weight filtration)
$$
0 = W_m \bE_\Q \subseteq W_{m+1} \bE_\Q \subseteq \cdots \subseteq W_n \bE_\Q = \bE_\Q
$$
of $\bE_\Q := \bE_\Z\otimes \Q$ by sub local systems.

\item A decreasing filtration (the Hodge filtration)
$$
\cE = F^a \cE \supseteq F^{a+1} \cE \supseteq \cdots \supseteq F^b \cE = 0
$$
of Deligne's canonical extension $\cE$ of $\bE_\Z\otimes \cO_S$ to $\Sbar$ by holomorphic sub-bundles. That is, $F^p\cE$ and $\cE/F^p\cE$ are locally free for each $p$.

\end{enumerate}

The flat connection on $\bE_\Z\otimes \cO_S$ extends to a connection
$$
\nabla : \cE \to \cE \otimes \Omega^1_\Sbar(\log D).
$$
Under the assumption of unipotent local monodromy, the canonical extension is characterized by the property that the residue of the connection at each smooth point of $D$ is nilpotent.

These data are required to satisfy the following conditions:
\begin{enumerate}[(a)]

\item The connection on $\cE$ satisfies Griffiths transversality:
\begin{equation}
\label{eqn:griff-trans}
\nabla : F^p \cE \to F^{p-1}\cE \otimes \Omega^1_\Sbar(\log D)
\end{equation}
for each $p$.

\item For each $s\in S$, the restriction of the Hodge and weight filtrations to the fiber of $\bE_\Z$ over $s$ is a mixed Hodge structure.

\item Each weight graded quotient $\Gr_W^m \bE$ admits a polarization in the sense of \cite[\S2]{schmid}.

\item For each holomorphically embedded arc $\alpha: \bD \to \Sbar$ satisfying $\alpha^{-1}(D) = \{0\}$, there is a {\em relative weight filtration}
$$
0 = M_r E_0 \subseteq M_{r+1} E_0 \subseteq \cdots \subseteq M_s = E_0
$$
of the fiber $E_0$ of $\cE$ over $\alpha(0)$. Set
$$
N = -\Res_0 \nabla|_\bD \in \End E_0.
$$
The relative weight filtration is characterized by the following properties:
\begin{enumerate}[(i)]

\item $N (M_j E_0) \subseteq M_{j-2} (E_0)$ and $N(W_r E_0) \subseteq W_r E_0$ for each $j,r \in \Z$.

\item for each $r,k \in \Z$, $N^k : M_{r+k} W_r E_0 \to M_{r-k} W_m E_0$ induces an isomorphism
\begin{equation}
\label{eqn:wt_iso}
\Gr^M_{r+k} \Gr^W_r E_0 \overset{\simeq}{\longrightarrow} \Gr^M_{r-k} \Gr^W_r E_0.
\end{equation}
See \cite{steenbrink-zucker} for more details and also \cite[\S7]{hain:tokyo} for an exposition with some relevant examples. Note that the filtration $M_\bdot$ is called the {\em monodromy weight filtration} when $\bE$ is a variation of Hodge structure.

\end{enumerate}
\end{enumerate}

These data determine, for each non-zero tangent vector $\vv \in T_0 \bD$, a canonical limit mixed Hodge structure which we denote by $E_\vv$. It is defined as follows. Choose a holomorphic parameter $t$ in $\bD$ such that $\vv = \del/\del t$. First observe that each weight graded quotient $\Gr^W_m \bE$ is a polarized variation of Hodge structure in the sense of Schmid \cite{schmid}. 

Standard  ODE implies that there is a unique trivialization
\begin{equation}
\label{eqn:trivialization}
\cE|_\bD \cong \bD \times E_0
\end{equation}
of the restriction of $\cE$ to $\bD$ in which the connection is
$$
\nabla = d - N\frac{dt}{t}.
$$
(This trivialization depends on the choice of parameter $t$.) The flat sections are of the form
$$
e(t) = t^N e \text{ where } e\in E_0.
$$
Identify the fiber $E_t$ of $\cE$ over $\alpha(t) \in \bD'$ with $E_0$ using the trivialization (\ref{eqn:trivialization}) above. This also allows us to regard $N$ as an endomorphism of $E_t$. With these identifications, we have
$$
h_t = \exp(2\pi i N)
$$
where $h_t : E_t \to E_t$ is the local monodromy operator.

The identification of the fiber $E_\lambda$ over $\lambda\in \bD'$ with $E_0$ defines a $\Z$ structure on $E_0$. Denote it by $E_{\lambda\vv,\Z}$. The limit mixed Hodge structure $E_{\lambda\vv}$ is the trifiltered vector space
$$
(E_0,F^\bdot,M_\bdot,W_\bdot)
$$
endowed with the lattice $E_{\lambda\vv,\Z}$. Schmid's result implies that each of its $W_\bdot$ graded quotients comprise a MHS with weight filtration $M_\bdot$.\footnote{One needs to check that $M_\bdot$ is defined over $\Q$ for each $
\lambda \in \bD'$. This follows from the fact that $h_\lambda$ is an automorphism of $E_{\Q,\lambda}$. See \cite{steenbrink-zucker}.} It then follows that for each $\lambda\in \bD'$, $E_{\lambda\vv}$ is a MHS with the filtration
$$
0 = W_m E_{\lambda\vv} \subseteq W_{m+1} E_{\lambda\vv} \subseteq \cdots \subseteq W_n E_{\lambda\vv} =  E_{\lambda\vv}
$$
by sub MHS. Although the identification of $E_\lambda$ with $E_0$ depends on the choice of parameter $t$, the $\Z$-structure (and thus the MHS) depends only on the tangent vector $\lambda\vv = \lambda\del/\del t \in T_0 \bD$.

In fact, the $\bE_{\lambda\vv}$ are defined for all $\lambda\neq 0$ and form a nilpotent orbit of MHS over the punctured tangent space $T_0'\bD := T_0 \bD - \{0\}$.

Griffiths transversality (\ref{eqn:griff-trans}) implies that $N (F^p E_0) \subseteq F^{p-1}E_0$. Since $N(M_rE_0) \subset M_{r-2}E_0$, we have:

\begin{proposition}
For each $\vv \in T_0'\bD$, the monodromy logarithm $N : E_0 \to E_0$ induces a morphism $N : E_\vv \to E_\vv(-1)$ of MHS.
\end{proposition}

\subsection{Nilpotent orbits}
\label{sec:nilp-orbits}

The 1-variable nilpotent orbit above is the restriction of a several variable nilpotent orbit. Such nilpotent orbits of limit MHS approximate admissible variations of MHS near a point of the boundary divisor $D$ and are useful for understanding the boundary behaviour of admissible variations.

Suppose that $U = \bD^n$ is a polydisk neighbourhood in $\Sbar$ of a point of $D$. Suppose that the intersection of $D$ with $U$ is defined by $t_1t_2\dots t_k = 0$, where $t_j$ is the coordinate in the $j$th disk. Let $D_j$ be the component of $D\cap U$ defined by $t_j = 0$. Denote the fiber of $\cE$ over the origin of $U$ by $E_0$. Set
$$
N_j = -\Res_{D_j,0} \nabla \in \End E_0
$$
Then $N_1,\dots,N_k$ are commuting nilpotent endomorphisms of $E_0$. There is a trivialization
\begin{equation}
\label{eqn:local-iso}
\cE|_U \cong U\times E_0
\end{equation}
of the restriction of $\cE$ to $U$ in which
\begin{enumerate}

\item $\nabla$ is given by
$$
\nabla = d - \sum_{j=1}^k N_j \frac{dt_j}{t_j}
$$

\item The isomorphism (\ref{eqn:local-iso}) restricts to an isomorphism of $W_m \cE|_U$ with $U\times W_m E_0$.

\end{enumerate}
Define a Hodge filtration on $U\times E_0$ by $F^p(U\times E_0) = U \times F^p E_0$. The fiber of $F^p \cE$ over $t\in U$ corresponds to a subspace $F^p_t E_0$ of $E_0$ via the isomorphism (\ref{eqn:local-iso}). Since both Hodge filtrations of the bundle $U\times E_0$ are holomorphic and agree at the origin, we have:

\begin{proposition}
\label{prop:nilp_orbit}
There is a holomorphic mapping $\Phi_\bE : U \to \Aut E_0$ with the property that:
\begin{enumerate}[(a)]

\item $\Phi_\bE(0) = \id_{E_0}$,

\item $\Phi_\bE(t)$ preserves $W_\bdot$ and the induced automorphism of $\Gr^W_0 E_0$ is the identity,

\item the original Hodge filtration on $\cE|_U$ is related to the Hodge filtration on $U\times E_0$ by
$$
F^p_t E_0 = \Phi_\bE(t)(F^p E_0).
$$
\end{enumerate}
\end{proposition}

Suppose that $x=(0,\dots,0,t_{k+1},\dots,t_n) \in \cap D_j$. For each tangent vector
$$
\vv = \sum_{j=1}^n \lambda_j \frac{\del}{\del t_j} \in T_x U
$$
that is not tangent to any component of $D$ --- equivalently, $\lambda_j \neq 0$ for $1\le j \le k$ --- we can restrict $\bE$ to the arc $t\mapsto (\lambda_1 t_1,\dots, \lambda_k t_k,t_{k+1},\dots,t_n)$ to obtain a limit MHS $E_\vv$ whose weight filtration is the relative weight filtration of 
$$
N = N_1 + N_1 + \dots + N_k.
$$
The $E_\vv$ form a nilpotent orbit of MHS over
$$
T_x U - \bigcup_{j=1}^k T_x D_j
$$
that is constant on the cosets of $T_x(\bigcap_{j=1}^k D_j)$. It therefore descends to a nilpotent orbit of MHS on the normal bundle of the intersection of the $D_j$ with the normal bundles of the $D_j$ removed. We shall denote this nilpotent orbit by $\bE^\nilp$. It is an admissible variation of MHS. The case $k=1$ plays an important role in the proof of Theorem~\ref{thm:max-rk}.

Nilpotent orbits are themselves admissible VMHS. They behave well under subquotients of admissible variations of MHS. In particular, the nilpotent orbit associated to the subquotient $\Gr^W_r \bE$ is $\Gr^W_r\bE^\nilp$.

\subsection{Admissible normal functions}
\label{sec:anf}

Suppose that $\bV$ is a polarizable variation of Hodge structure of negative weight over a smooth variety $S$. An admissible normal function is a section $\nu$ of $J(\bV)$ that corresponds to an admissible variation of MHS $\bE$ over $S$ that is an extension of $\Z_S(0)$ by $\bV$.

\begin{proposition}
\label{prop:anf}
Normal functions associated to families of homologically trivial cycles are admissible.
\end{proposition}

\begin{proof}
Suppose that we are in the setting of Section~\ref{sec:families_cycles}. Let $\bE$ be the extension $\bE_Z$ associated to $Z$. There is a Zariski open subset $U$ of $S$ on which the map $(X,|Z|) \to S$ is topologically locally trivial.  To prove $\bE$ is an admissible variation of MHS over $U$ it suffices, by \cite{kashiwara}, to show that its restriction to each curve in $S$ is admissible. This follows from the main result of \cite{steenbrink-zucker} applied to $(X-|Z|)|_U \to U$. It also follows by appealing to Saito's theory of mixed Hodge modules \cite{saito}. The result of \cite[\S7]{hain:msri} implies that $\bE$ is an admissible variation over $S$.
\end{proof}

\section{Codimension 1 degenerations of normal functions}
\label{sec:boundary_normal}

We now study codimension 1 degenerations of admissible normal functions in the case where $\bV$ is an admissible variation of Hodge structure of weight $-1$. In this case, $J(\bV)$ is a holomorphic family of compact complex tori.

\subsection{Setup}
\label{sec:setup}

Suppose that $\Sbar$ is a smooth variety (not necessarily compact) and that $\Delta$ is a connected smooth divisor in $\Sbar$. Set $S=\Sbar-\Delta$. Suppose that $\bV$ is a polarizable variation of Hodge structure of weight $-1$ over $S$ and that $\nu : S \to J(\bV)$ is an admissible normal function. Denote the corresponding variation of MHS by $\bE$. We shall assume that $\bV_\Z$ is torsion free, a condition that is satisfied by all variations in Part~\ref{part:higher-genus}.

Denote the normal bundle of $\Delta$ in $\Sbar$ by $L$ and the complement of its 0-section, which we identify with $\Delta$, by $L'$. Let $\pi : L \to \Delta$ be the projection and $\pi'$ its restriction to $L'$. By the discussion in Section~\ref{sec:nilp-orbits}, there are nilpotent orbits of MHS $\bE^\nilp$ and $\bV^\nilp$ over $L'$ and an extension
\begin{equation}
\label{eqn:nilpt-extn-Delta}
0 \to \bV^\nilp \to \bE^\nilp \to \Q\ee \to 0
\end{equation}
of variations over $L'$. Both $\bV^\nilp$ and $\bE^\nilp$ are filtered by subvariations $M_j\bV^\nilp$ and $M_j\bE^\nilp$.

Even though $\bV^\nilp$ is not a variation of Hodge structure of pure weight $-1$ (except when $N=0$), we can still define $J(\bV^\nilp)$ to be the (not necessarily topologically locally trivial) family of complex tori over $L'$ whose fiber over $\vv \in L_x'$ is
$$
J(V_\vv) := V_{x,\C}/(V_{\vv,\Z} + F^0 V_x).
$$
Schmid's result \cite[Thm.~4.9]{schmid} implies that there is a neighbourhood $U$ of the zero section of $L$ such that $\bV^\nilp$ is a polarizable variation of HS of weight $-1$ over $U' := U - \Delta$. Consequently, $J(\bV^\nilp)$ is a family of compact complex tori over $U'$. Denote by $\nu^\nilp : U' \to J(\bV^\nilp)$ the normal function section of it that corresponds to the restriction of (\ref{eqn:nilpt-extn-Delta}) to $U'$.

Since the monodromy logarithm $N : \bV^\nilp \to \bV^\nilp(-1)$ is a morphism, $\ker N$ is a subvariation of $M_{-1}\bV^\nilp$. Since it has trivial monodromy, it is constant on each $L_x'$ and thus extends to a variation over $L$, which we denote by $\bK$. Since it is constant on each $L_x$, it is the pullback along $\pi$ of its restriction $\bK_0$ to $\Delta$.

\subsection{The N\'eron model associated to a nilpotent orbit}

Here we give the construction of the N\'eron model for $\bV^\nilp$ on the punctured neighbourhood $U'$ of the zero section of $L$. It is more general than the construction given in \cite{ggk} as we allow the base to have dimension $>1$, but it is less general in that we construct N\'eron models only for nilpotent orbits. Working with nilpotent orbits is sufficient for our purposes and simplifies the construction as nilpotent orbits are filtered by their relative weight filtrations.

Regard $J(\bV_\R)$ as a local system of abelian groups over $L'$.

\begin{lemma}
\label{lem:obstruction}
For each $x\in \Delta$, there is an exact sequence
$$
0 \to J(\bK_\R)|_{L_x}  \to H^0(L_x',J(\bV_\R)) \to G_x \to 0,
$$
where $G_x$ is the finite abelian group
$$
G_x := \text{torsion subgroup of }H_0(L_x',\bV_\Z|_{L_x'}).
$$
In more concrete terms
$$
G_x = \big(V_{\vv,\Z} \cap (h-1)V_{\vv,\Q} \big)/(h - 1) V_{\vv,\Z},
$$
where $\vv \in L_x'$ and $h : V_{\vv,\Z} \to V_{\vv,\Z}$ is the local monodromy operator.
\end{lemma}

\begin{proof}
Denote $\bV\otimes_\Z \kk$ by $\bV_\kk$, where $\kk$ is a subring of $\R$. Since $V_\Z$ is torsion free, the long exact sequence of cohomology associated to the exact sequence
$$
0 \to \bV_\Z \to \cV_\R \to J(\bV_\R) \to 0
$$
of local systems over $L_x'$
$$
0 \to H^0(L_x',\bV_\Z) \to H^0(L_x',\cV_\R) \to H^0(L_x',J(\bV_\R)) \to H^1(L_x',\bV_\Z) \to H^1(L_x',\bV_\R).
$$
Since $H^1(L_x',\bV_\R)$ is torsion free and since $H^0(L_x',\bV_\kk) = K_\kk$, the sequence
$$
0 \to K_\Z \to K_\R \to H^0(L_x',J(\bV_\R)) \to H^1(L_x',\bV_\Z)^\tor \to 0
$$
is exact, where $K_\kk$ denotes the fiber of $\bK_\kk$ over $x\in \Delta$. Since $L_x'$ has the homotopy type of a circle, we have the Poincar\'e duality isomorphism $H^1(L_x',\bV_\Z) \cong H_0(L_x',\bV_\Z)$. It restricts to an isomorphism on torsion subgroups. The results now follows as $H^0(L_x',J(\bK)) = K_\R/K_\Z$.
\end{proof}

The groups $G_x$ form a local system over $L$. Denote it by $\bG$. Denote the local system over $L$ whose restriction to $L_x$ is $H^0(L_x',\bV)$ by $\bT$. We have a short exact sequence of local systems
$$
0 \to \bK \to \bT \to \bG \to 0.
$$
All terms in this sequence are admissible variations of MHS over $L$. The quotient variation $\bG$ is Tate of weight 0. Define $J(\bT)$ by defining its fiber over $x\in \Delta$ to be
$$
J(T_x) = T_{x,\C}/(K_{x,\Z} + F^0 K_x).
$$
There is an exact sequence
$$
0 \to J(\bK) \to J(\bT) \to \bG \to 0
$$
over $L$. The inclusion $\bT \to \bV^\nilp$ induces a morphism
$$
J(\bT)|_{L'} \to J(\bV^\nilp)
$$
over $L'$. The N\'eron model $\Jtilde(\bV^\nilp)$ of $J(\bV^\nilp)$ is the family of abelian complex Lie groups over $U$ obtained by glueing $J(\bV^\nilp)|_{U'}$ to $J(\bT)|_U$ by identifying $J(\bT)|_{U'}$. In other words
$$
\begin{tikzcd}
J(\bT)|_{U'} \ar[r] \ar[d] & J(\bV^\nilp)|_{U'} \ar[d] \\
J(\bT)|_U \ar[r] &  \Jtilde(\bV^\nilp)
\end{tikzcd}
$$
is a pushout square. Its restriction to $U'$ is $J(\bV^\nilp)|_{U'}$ and its fiber over the zero section is $J(\bT)|_\Delta$. It is a {\em separated slit analytic space}. (See \cite[p.~308]{ggk} for the definition.) If $N^2 \neq 0$, $\Jtilde(\bV^\nilp)$ is not a manifold.

\begin{remark}
\label{rem:modular}
If $S$ is a modular curve and if $\bV$ is the local system $R^1 f_\ast \Z(1)$ associated to the universal elliptic curve $f : \cE \to S$, then $J(\bV)$ is the universal elliptic curve. The restriction of $\bV$ to a punctured neighbourhood $U'$ of a cusp $\Delta$ is a nilpotent orbit. So $J(\bV) = J(\bV^\nilp)$. The N\'eron model $\Jtilde(\bV)$ over $U$ is the usual N\'eron model.

This example illustrates another feature of the construction. The family $J(\bK)$ over $U$ is the trivial $\C^\times$ bundle and, with the correct choice of coordinates, the fiber of $J(\bV)$ over $q\in U'$ is $\C^\times/q^\Z$. In particular, the map $J(\bK) \to J(\bV)$ has infinite degree. In this case it is surjective as $N^2 = 0$.
\end{remark}

The following result is closely related to the main results of \cite{ggk} and \cite{schnell}.

\begin{theorem}
\label{thm:extension}
The normal function $\nu^\nilp$ lifts to a holomorphic section $\nutilde^\nilp$ of $J(\bT)$ over $U$ and this section projects to a holomorphic section $\nuhat^\nilp$ of $\Jtilde(\bV^\nilp)$ whose restriction to $U'$ is $\nu^\nilp$.
$$
\begin{tikzcd}
J(\bT) \ar[r]\ar[d] & \Jtilde(\bV^\nilp) \ar[d] \ar[r,hookleftarrow] & J(\bV^\nilp) \ar[d] \\
U \ar[r,equal] \ar[u, bend right=40, "\nutilde^\nilp"'] & U \ar[u, bend right=40, "\nuhat^\nilp"'] \ar[r,hookleftarrow] & U' \ar[u, bend right=40, "\nu^\nilp"']
\end{tikzcd}
$$
\end{theorem}

\begin{proof}
We work locally in $L$. Let $X$ be a contractible open Stein subset of $\Delta$. Set
$$
Z = U \cap \pi^{-1}(X),
$$
where $\pi : L \to \Delta$ denotes the projection, and $Z' = Z\cap U'$.  Recall that the restriction of $\cE$ (and thus $\cV$) to each $L_x$ is trivial, so we can (and will) identify their fibers $E_z$ and $V_z$ over $z \in Z$ with $E_x$ and $V_x$, respectively, where $z\in L_x'\cap U$.

Exactness of $F^0$ and $M_0$ implies that there is a holomorphic section $\ee_F$ of $F^0 M_0 \cE$ over $Z$ that is constant (with respect to the trivialization) on each $L_x'$ and projects to the section $1$ of $\Gr^W_0\bE_\Z \cong \Z$. There is also a multivalued section $\ee_\Z$ of $\bE_\Z$ over $Z'$ that projects to $1\in H^0(Z',\Gr^W_0 \bE^\nilp$). Then 
$$
\bv := \ee_\Z - \ee_F
$$
is a multivalued holomorphic section of $\cV$ over $Z'$ which descends to $\nu^\nilp \in H^0(Z',J(\bV^\nilp))$. To prove the theorem, we will show that it lifts to a section of $J(\bT)$ over $Z'$.

Denote the local monodromy operator by $h$ and its logarithm by $N$. Since $h$ is unipotent, the endomorphisms $(h-1)$ and $N$ of $E_{z,\kk}$ satisfy
$$
\im N = \im (h-1) \text{ and } \ker N = \ker (h-1)
$$
when $\kk$ is a subfield of $\C$. The property (\ref{eqn:wt_iso}) of the relative weight filtration implies that
$$
\ker\{N : V_{z,\kk} \to V_{z,\kk}\}  \subseteq M_{-1} V_{z,\kk},\quad \im\{ (h-1) : V_{z,\kk} \to V_{z,\kk}\} \supseteq  M_{-2} V_{z,\kk}.
$$
Since $N$ is a morphism of type $(-1,-1)$,
$$
\im \{ N : F^0 M_0 E_z \to V_z\} \supseteq F^{-1} M_{-2} V_z.
$$
So $N \ee_F$ is a section of $F^{-1}M_{-2}\cV$. The properties of the relative weight filtration imply that there is a holomorphic section $\ee'$ of $F^0 M_0 \cV$ over $Z$ such that $N\ee_F = N\ee'$. It is constant (with respect to the trivialization) on each $L_x'\cap U$. By replacing $\ee_F$ by $\ee_F-\ee'$, we may assume that $N\ee_F = 0$. That is, the restriction of $\ee_F$ to each $L_x$ is also a horizontal section of $\cK$.

The class of $\nu^\nilp$ in $H^0(X,\bG)$ is represented by $(h-1)\ee_\Z(z)$ in
$$
G_x = \big(V_{z,\Z} \cap (h-1)V_{z,\Q} \big)/(h - 1) V_{z,\Z} \quad \text{for all } z \in L_x'\cap U.
$$
If it vanishes at one (and hence all) $x\in X$, there is a multivalued section $\bv_\Z$ of $\bV_\Z$ over $Z'$ such that $(h-1)\bv_\Z = (h-1)\ee_\Z$. By replacing $\ee_\Z$ by $\ee_\Z - \bv_\Z$, we see that if the class of $\nu^\nilp$ vanishes in $H^0(Z',\bG)$, then we may choose $\ee_\Z$ so that $(h-1)\ee_\Z = 0$. That is, $\ee_\Z$ extends to a section of $\bK_\Z$ over $Z$. Since $\ee_F$ extends to a holomorphic section of $\cK$ over $Z$, $\bv = \ee_\Z - \ee_F$ also extends to a holomorphic section of $\cK$ over $Z$. It is constant on each slice $L_x'\cap U$ and descends to a holomorphic section of $J(\bK)$ over $Z$ whose restriction to $Z'$ is a lift of $\nu^\nilp$. This completes the proof when the obstruction in $H^0(X,\bG)$ vanishes. We now prove the general case.

Every element of $H^0(U',J(\bV))$ extends to $U$. Since $J(\bV) \cong J(\bV_\R)$, every element of $G_x$ is the class of a constant section of $J(\bV)$ over $U\cap L_x'$. All such sections lift to sections of $\Jtilde(\bT)$ over $U\cap L_x$. So, if the class of $\nu^\nilp$ in $H^0(X,\bG)$ is non-zero, we can find a constant section $\nubar$ of $J(\bV)$ with the same class in $G_x$. Since the class of $\nu^\nilp-\nubar$ vanishes, by the previous paragraph, it lifts to a holomorphic section of $J(\bK)$ over $U\cap L_x$ for all $x \in X$. Since $\nubar$ lifts to a section of $\Jtilde(\bT)$, it follows that $\nu^\nilp$ lifts to a section $\nuhat^\nilp$ of $\Jtilde(\bT)$ over $Z$ whose restriction to $Z'$ projects to $\nu^\nilp$. 
\end{proof}

\begin{remark}
The N\'eron model $\Jtilde(\bV)$ of $J(\bV)$ can be constructed from $\Jtilde(\bV^\nilp)$ using the Nilpotent Orbit Theorem in the guise of Proposition~\ref{prop:nilp_orbit}. Just use the function $\Phi$ to perturb the Hodge filtration of $\bV^\nilp$ to obtain the Hodge filtration of $\bV$ to construct $J(\bV)$ from $J(\bV^\nilp)$. The restrictions of $\Jtilde(\bV)$ and $\Jtilde(\bV^\nilp)$ to $\Delta$ are equal. Similarly, the Nilpotent Orbit Theorem implies that the normal function section $\nu$ of $J(\bV)$ extends to a section $\nuhat$ of $\Jtilde(\bV)$ and that the restrictions of $\nuhat$ and $\nuhat^\nilp$ to $\Delta$ are equal.
\end{remark}

Since $\Sbar-S$ is a smooth connected divisor $\Delta$, the local system $\bV$ extends to a local system on $\Sbar$ if and only if $N=0$. In this case, the variation $\bV$ extends to a polarized variation of Hodge structure over $\Sbar$. The fiber over $x\in\Delta$ is the limit MHS associated to any tangent vector $\vv$ of $\Sbar$ at $x$ that is not tangent to $\Delta$. The limit does not depend on the choice of $\vv$.

\begin{corollary}[{\cite[Thm.~7.1]{hain:msri}}]
\label{cor:extn}
If $N=0$, then $\nu$ extends to a holomorphic section of $J(\bV)$ over $\Sbar$.
\end{corollary}

The derivative of the lift $\nutilde^\nilp$ of $\nu^\nilp$ constructed in Theorem~\ref{thm:extension} is a section of $F^{-1}\cV\otimes \Omega^1_\Sbar(\log \Delta)$. Combined with the Nilpotent Orbit theorem (Proposition~\ref{prop:nilp_orbit}), this implies that the Green--Griffiths invariant of an admissible normal function is logarithmic along the smooth boundary divisor $\Delta$.

\begin{corollary}
\label{cor:gg-log}
The Green--Griffiths invariant $\deltabar(\nu)$ of $\nu$ extends to a holomorphic section of
$$
\Gr_F^{-1}\cV \otimes \Omega^1_\Sbar(\log\Delta)/\nablabar \Gr_F^0 \cV.
$$
over $\Sbar$.
\end{corollary}

\section{Residual normal functions and the normal rank}

The nilpotent Orbit Theorem (Proposition~\ref{prop:nilp_orbit}) implies that the rank of $\nu$ on $S$ is bounded below by the rank of $\nu^\nilp$ on $L'$. In this section we introduce several tools for estimating the rank of $\nu^\nilp$. The first is the normal rank.

\begin{definition}
\label{def:normal_rk}
The {\em normal rank} $\rank^\perp_x\nu$ of $\nu$ at $x\in \Delta$ is defined to be the rank of the restriction of $\nu^\nilp$ to $L_x'$. The normal rank $\rank_\Delta^\perp \nu$ of $\nu$ along $\Delta$ is defined by
$$
\rank_\Delta^\perp \nu := \max_{x\in \Delta} \rank^\perp_x\nu \in \{0,1\}.
$$
\end{definition}

The second tool is the {\em residual normal function} $\nu_\Delta$, defined below.

\subsection{The residual  normal function $\nu_\Delta$}

By Theorem~\ref{thm:extension}, $\nu^\nilp$ has a natural lift $\nutilde^\nilp$ to a section of $J(\bT)$ over a neighbourhood $U$ of $\Delta$ in $L$. The order of the class of $\nu^\nilp$ in $H^0(U,\bG)$ is the least positive integer such that $k\nutilde^\nilp$ is a section of $J(\bK)$. Since $\bK \subseteq M_{-1}\bV^\nilp$, $k\nutilde^\nilp$ descends to a section of $J(\Gr^M_{-1}\bV^\nilp)$ over $U$. Set
$$
\bV_\Delta := \Gr^M_{-1} \bV^\nilp.
$$
This is a polarizable variation of Hodge structure over $\Delta$ of weight $-1$.

\begin{definition}
\label{def:residual_nf}
The residual normal function $\nu_\Delta$ is the restriction to $\Delta$ of the section of $J(\bV_\Delta)$ determined by $k\nutilde^\nilp$.
\end{definition}

\begin{remark}
\label{rem:biext}
The section $k\nutilde^\nilp$ of $J(\bK)$ over $\Delta$ corresponds to an extension
$$
0 \to \bK \to \bL \to \Z \to 0
$$
of admissible variations of MHS over $\Delta$, which we call the {\em residual variation}. This extension, after tensoring with $\Q$, is a subquotient of $M_0\bV^\nilp$. The variation that corresponds to the residual normal function $\nu_\Delta$ is $\bL/M_{-2}$.
\end{remark}

\subsection{The real N\'eron model}

The real local system $\bT_\R$ underlying $\bT$ has constant fiber $H^0(L_x',\bV_\R)$ over $L_x$ for each $x\in \Delta$. Define
$$
J(\bT_\R) = \bT_\R/\bT_\Z.
$$
This is a family of compact real tori over $L$ which fits in the exact sequence
$$
0 \to J(\bK_\R) \to J(\bT_\R) \to \bG \to 0.
$$
The inclusion $J(\bT_\R) \hookrightarrow J(\bT)$ induces an isomorphism on fundamental groups.

Even though $J(\bT)|_{U'} \to J(\bV^\nilp)$ has infinite degree when $N\neq 0$, the map
$$
J(\bT_\R) \to J(\bT) \to \Jtilde(\bV^\nilp)
$$
is an inclusion.

\begin{remark}
\label{rem:real-jac}
The variation $\bK$ is a constant subvariation of $\bV^\nilp$ with strictly negative weights. Denote its fiber by $K$. Since $K_\R \cap F^0 K = 0$, the torus $J(K_\R) := K_\R/K_\Z$ is the kernel of the map
$$
\Ext^1_\MHS(\Z,K) \to \Ext^1_\MHS(\R,K) = K_\C/(K_\R + F^0 K).
$$
Thus $J(K_\R)$ consists of those extensions of $\Z$ by $K$ that split as an $\R$-MHS.
\end{remark}

\begin{remark}
This is a continuation of Remark~\ref{rem:modular}. In the situation described there, $J(\bK_\R)$ is the trivial $S^1$ bundle over $\bD$. Its image in $J(\bK)$, the trivial $\C^\times$ bundle, is the unit circle. It is the kernel of the map
$$
\Ext^1_\MHS(\Z,\Z(1)) \to \Ext^1_\MHS(\R,\R(1)).
$$
The group $G_\Delta$ is the group of torsion sections of $J(\bV)$ modulo those that lie in the circle bundle $J(\bK_\R)$. The bundle $J(\bT_\R)$ is $G_x \times J(\bK_\R)$. Its image in a fiber of $J(\bV)$ is the real subgroup whose identity component is the image of the unit circle in $\C^\times$ and its translates by torsion sections that correspond to the non-trivial elements of $G_x$.
\end{remark}

\subsection{The normal rank}

To study the normal rank of $\nu$ at $x\in \Delta$, we restrict the variations to the fiber $L_x'$ of $L'$ over $x$. The open neighbourhood $U$ of $\Delta$ in $L$ will be as in Section~\ref{sec:setup}. Recall that $U'=U-\Delta$.

\begin{lemma}
\label{lem:rank0}
Let $k>0$ be the order of the class of $\nu^\nilp$ in $H^0(U\cap L_x',\bG)$. The restriction of the normal function $\nu^\nilp$ to $U\cap L_x'$ has rank 0 if and only if the restriction of $k\nuhat^\nilp$ to $U\cap L_x'$ factors through the inclusion $\Jtilde(\bK_\R)|_{U\cap L_x'}  \hookrightarrow \Jtilde(\bV^\nilp)|_{U\cap L_x'}$.
\end{lemma}

\begin{proof}
First note that multiplication of a normal function by a positive integer does not change the rank. By replacing $\nu^\nilp$ by $k\nu^\nilp$, we may assume that $\nutilde^\nilp$ is a section of $J(\bK)$. Then $\nu^\nilp$ has rank 0 over $U\cap L_x'$ if and only if it lifts to a constant section of $\bV_\R$ over $U\cap L_x'$. In other words, it lifts to an element of $H^0(L_x',\bK_\R)$ and thus descends to a section of $\Jtilde(\bK_\R)$ over $L_x'$.
\end{proof}

Suppose that $\bD$ is an analytic disk in $\Sbar$ that intersects $\Delta$ transversally at $x$. Set $\bD'= \bD\cap L'$. Observe that the nilpotent orbit associated to $\bV|_{\bD'}$ is $\bV^\nilp|_{L_x'}$.

\begin{corollary}
We have $0 \le \rank \nu^\nilp_{U\cap L_x'} \le \rank \nu|_{\bD'} \le 1$.
\end{corollary}

\begin{proof}
By multiplying $\nu|_{\bD'}$ and $\nu^\nilp|_{L_x'}$ by a positive integer, we may assume that $\nuhat^\nilp|_{U\cap L_x'}$ lifts to a section of $J(\bK)$. It suffices to show that if $\nu|_{\bD'}$ has rank 0, then so does $\nu^\nilp|_{U\cap L_x'}$. If $\nu|_{\bD'}$ has rank 0, then $\nu(0) \in J(K_\R)$, where $K$ is the fiber of $\bK$ over $x$. Since $\nu(0) = \nuhat^\nilp(0)$, this implies that $\nu^\nilp|_{U\cap L_x'}$ is a section of $J(\bK)$. Lemma~\ref{lem:rank0} now implies that $\nu^\nilp|_{U\cap L_x'}$ also has rank 0.
\end{proof}

\subsection{The residual normal function $\nu_\Delta$ and the rank of $\nu$}

\begin{proposition}
\label{prop:ranks}
With the notation above, we have
\begin{enumerate}

\item $\rank \nu \ge \rank \nu^\nilp$,

\item
\label{item:inequality}
$\rank \nu^\nilp \ge \rank \nu_\Delta + \rank_\Delta^\perp \nu$,

\item $\rank_\Delta^\perp \nu = 0$ if and only if there is a  positive integer $k$ such that $k\nu^\nilp$ is a section of $J(\bK_\R)$ that is constant on each fiber of $L' \to \Delta$.

\end{enumerate}
\end{proposition}

Note that equality may not hold in (\ref{item:inequality}) when $\pi_1(\Delta,x)$ acts non-trivially on the kernel of the restriction mapping $H_0(L_x,W_0\End V_\vv) \to \End K_x$.

\begin{proof}
As previously remarked, the first assertion follows from the Nilpotent Orbit Theorem (Prop.~\ref{prop:nilp_orbit}). The second assertion follows from the commutativity of the diagram
$$
\begin{tikzcd}[column sep=scriptsize]
0 \ar[r] & T_\vv L_x' \ar[r] \ar[d,"\nablabar\nu^\nilp|_{L_x}"] & T_\vv L' \ar[r,"\pi_\ast"] \ar[d,"\nablabar\nu^\nilp"] & T_x \Delta \ar[r] \ar[d,"\nablabar\nu_\Delta"] & 0 \\
0 \ar[r] & \Gr_F^{-1} M_{-2} V_x \ar[r] & \Gr_F^{-1} M_{-1} V_x \ar[r] & \Gr_F^{-1} \Gr_{-1}^M V_x \ar[r] & 0
\end{tikzcd}
$$
where $x\in \Delta$ and $\vv \in L_x'$. The third assertion follows from Lemma~\ref{lem:rank0}.
\end{proof}

\begin{corollary}
\label{cor:splitting}
If $\rank \nu^\nilp = \rank \nu_\Delta$, then there is a normal function section $\nubar$ of $J(\bK_\R)$ over $\Delta$ and a positive integer $k$ such that the diagram
$$
\begin{tikzcd}
 & J(\bK_\R) \ar[d] \\
L' \ar[ur,"k\nu^\nilp"] \ar[r,"\pi"] & \Delta \ar[u, bend right=40, "\nubar"']
\end{tikzcd}
$$
commutes. Consequently, each fiber of the extension of $\Z$ by $\bK$ (as variations of MHS) determined by $k\nu^\nilp$ splits after tensoring with $\R$.
\end{corollary}

\begin{proof}
If $k\nu^\nilp$ is a section of the subtorus $J(\bK_\R)$ of $J(\bV^\nilp)$, it is the pullback along $L' \to \Delta$ of a section of $J(\bK_\R)$ over $\Delta$ as it is constant over each $L_x'$. The second assertion follows from Remark~\ref{rem:real-jac}.
\end{proof}

\part{Higher genus}
\label{part:higher-genus}

In this part we use the results of Part~\ref{part:vmhs} and induction to prove that the Ceresa normal function has rank $3g-3$ when $g\ge 4$. The basic idea behind the proof is to approximate the genus $g$ Ceresa normal function by its nilpotent orbit on the normal bundle of the smooth points $\Delta$ of the divisor $\Delta_0$ in $\cMbar_g$. A global monodromy computation allows us to compute the residual normal function. This and the inductive hypothesis imply that the residual normal function $\nu_\Delta$ has rank $3g-4$. The proof is completed by using Corollary~\ref{cor:splitting} to show that the normal rank is 1, which will establish the result.

\section{Behaviour of the Ceresa normal function near $\Delta_0$}
\label{sec:higher_genus}

Suppose that $g > 3$. We will apply the results and constructions of Section~\ref{sec:boundary_normal} with
$$
S = \cM_g,\ \Sbar = \cMbar_g - (\Delta_+\cup \Delta_0^\sing) \text{ and } \Delta = \Delta_0 \cap \Sbar,
$$
where $\Delta_+$ is the union of the boundary divisors $\Delta_j$ where $j>0$. Recall that $L$ is the normal bundle of $\Delta$ in $\cM_g$ and that $L'$ is $L$ with its 0-section removed.

\subsection{A moduli description of $L'$}

A point of $\Delta$ corresponds to a smooth projective curve $C$ of genus $g-1$ together with an {\em unordered} pair $\{p,q\}$ of distinct points of $C$. We will identify $\Delta$ with the quotient of the moduli space $\cM_{g-1,2}$ that parameterizes smooth projective curves of genus $g-1$ with two distinct marked points. The involution $\sigma$ of $\cM_{g-1,2}$ swaps the two points.

The moduli space $\cM_{g-1,\vec{2}}$ is the $\Gm\times\Gm$ torsor over $\cM_{g-1,2}$ whose fiber over $(C;p,q)$ is $T_p' C \times T_q' C$ --- ordered pairs $(\vv_p,\vv_q)$ of non-zero tangent vectors of $C$ at $p$ and $q$. The group $\Gm$ acts on $\cM_{g-1,\vec{2}}$ by
$$
\lambda : (C;\vv_p,\vv_q) \mapsto (C;\lambda \vv_p,\lambda^{-1} \vv_q),\quad \lambda \in \Gm(\C).
$$
A point of $\cM_{g-1,\vec{2}}/\Gm$ corresponds to a 4-tuple $(C;p,q,\vv)$, where $C$ is smooth and projective of genus $g-1$, $(p,q)$ is an {\em ordered} pair of distinct points of $C$, and $\vv$ is a non-zero element of $T_p C\otimes T_q C$.

By deformation theory, there is an \'etale double covering (in the sense of stacks)
$$
\cM_{g-1,\vec{2}}/\Gm \to L'.
$$
The automorphism group of the covering is generated by the involution $\sigma$ that swaps the two points $p$ and $q$.

The pullback of the normal function of the genus $g-1$ Ceresa cycle along $\cM_{g-1,\vec{2}} \to \cM_{g-1}$ descends to a normal function section of $J(\Lambda^3_0 \bH_\Delta)$ over $\Delta$. Denote it by $\nu_0$. We also have the normal function sections $\kappa_p$ and $\kappa_q$ of $J(\bH_\Delta)$ defined by
$$
\kappa_p : (C;\vv_p,\vv_q) \mapsto (2g-4)[p] - K_C \text{ and } \kappa_p : (C;\vv_p,\vv_q) \mapsto (2g-4)[q] - K_C,
$$
where $K_C$ denotes the canonical divisor class of $C$.

\begin{proposition}
\label{prop:ext}
If $g\ge 4$, then
$$
\Ext^1_{\MHS(\Delta)}(\Q,\Lambda^3 \bH_\Delta) \cong \Ext^1_{\MHS(L')}(\Q,\Lambda^3 \bH_\Delta) = \Q\nu_0 \oplus \Q\kappa
$$
where $\kappa := (\kappa_p + \kappa_q)/2$. That is,
$$
\kappa(C;\{p,q\}) = \kappa(C;\vv_p\otimes\vv_q) = (g-2)(p+q) - K_C \in \Jac C.
$$
\end{proposition}

\begin{proof}
Since $g\ge 4$, $\Delta$ and $L'$ are quotients of moduli spaces of curves of genera $\ge 3$. So we can apply the results of Appendix~\ref{app:normal}. More precisely, we see that (after tensoring with $\Q$) the normal function sections of $J(\Lambda^3 \bH_\Delta)$ over $\Delta$ (resp. $L'$) are the normal function sections of its pullback that are invariant under the involution that swaps $p$ and $q$. Since, after tensoring with $\Q$, $\Lambda^3 \bH_\Delta(-1)$ is isomorphic to the direct sum of $\bH_\Delta$ and $\Lambda^3_0 \bH_\Delta(-1)$, the result follows from the classification result in the appendix.
\end{proof}

\subsection{The family of nilpotent orbits over $L'$}

Denote by $\bV^\nilp$ the family of nilpotent orbits over $L'$ associated to the $\Z$ variation $\bV := (\Lambda^3 \bH/\theta\cdot \bH)(-1)$ over $\cM_g$. Denote the fiberwise nilpotent monodromy operator $\bV^\nilp \to \bV^\nilp$ by $N$. It satisfies $N^2 = 0$. Denote the associated relative weight filtration by 
$$
0 \subset M_{-2} \bV^\nilp \subset M_{-1} \bV^\nilp \subset M_0 \bV^\nilp.
$$

Suppose that $(C;p,q,\vv)$ is in $\cM_{g-1,\vec{2}}/\Gm$. Set $C'= C-\{p,q\}$. Denote the homology class of a small positive loop about $p$ in $C'$ by $\ba$ (this is the ``vanishing cycle''). Choose a path in $C$ from $q$ to $p$ and denote its class in $H_1(C,\{p,q\})$ by $\bb$. It is well defined mod $H_1(C)$.

Denote the first order smoothing associated to $\vv$ of the nodal curve obtained from $C$ by identifying $p$ and $q$ by $C_\vv$. This has a limit MHS; it is the fiber of $\bH^\nilp$ over $\vv \in L'$. The monodromy logarithm $N$ takes $\bb$ to $\ba$. The graded quotients of the relative weight filtration are:
$$
\Gr^M_{-2} H_1(C_\vv) = \Q\ba,\ \Gr^M_{-1} H_1(C_\vv) = H_1(C),\ \Gr^M_0 H_1(C_\vv) = \Q\bb.
$$
The involution $\sigma$ that swaps $p$ and $q$ acts on $H_1(C_\vv)$ by
$$
\ba \mapsto -\ba \text { and } \bb \mapsto -\bb \bmod M_{-1}H_1(C_\vv).
$$

Denote the weight $-1$ polarized variation of HS over $L'$ with fiber $H_1(C)$ over $(C;\{p,q\},\vv)$ by $\bH_\Delta$. It is pulled back from $\Delta$. The $M_\bdot$ graded quotients of $\bH^\nilp$ over $\cM_{g-1,\vec{2}}/\Gm$ are
$$
\Gr^M_{-2} \bH^\nilp = \Q\ba ,\ \Gr^M_{-1}\bH^\nilp = \bH_\Delta,\ \Gr^M_{0} \bH^\nilp = \Q\bb.
$$
Denote the symplectic forms of $\bH_\Delta$ and $\Gr^M_\bdot\bH^\nilp$ by $\theta_\Delta \in \Lambda^2 \bH_\Delta$ and $\theta \in \Gr^M_{-2}\Lambda^2 \bH$, respectively. We have
$$
\theta = \theta_\Delta + \ba \wedge \bb \text{ in } \Gr^M_{-2} \Lambda^2 \bH^\nilp.
$$
Set $\Lambda^2_0 \bH_\Delta = \Lambda^2 \bH_\Delta/\theta_\Delta$. It is an irreducible local system over $\Delta$ and $L'$.

\begin{proposition}
\label{prop:GrM_V}
The $M_\bdot$ graded quotients of the pullback of $\bV^\nilp$ to the double cover $\cM_{g-1,\vec{2}}/\Gm$ of $L'$ are:
$$
\Gr^M_{-2} \bV = \ba \cdot \Lambda^2_0 \bH_\Delta(-1),\ \Gr^M_{-1} \bV \cong \Lambda^3 \bH_\Delta(-1),\ \Gr^M_{0} \bV = \bb \cdot \Lambda^2_0 \bH_\Delta(-1).
$$
\end{proposition}

Here we are regarding $\bH_\Delta$ as a variation over $L'$ by pulling it back along the projection $L' \to \Delta$.

\begin{proof}
Since
$$
\theta \cdot \Gr^M_j\bH^\nilp \cong
\begin{cases}
\Q\bb\cdot\theta_\Delta & j=-2,\cr
\theta \cdot \bH_\Delta & j=-3, \cr
\Q\ba\cdot\theta_\Delta & j=-4
\end{cases}
$$
we have
$$
\Gr^M_j \Lambda^3 \bH^\nilp =
\begin{cases}
\Q\bb\cdot\Lambda^2 \bH_\Delta & j=-2,\cr
\Lambda^3 \bH_\Delta + (\ba\wedge\bb) \cdot \bH_\Delta & j=-3, \cr
\Q\ba\cdot\Lambda^2 \bH_\Delta & j=-4.
\end{cases}
$$
The inclusion $\theta\cdot \Gr^M_{-1}\bH^\nilp \hookrightarrow \Lambda^3 \Gr^M_{-1}\bH^\nilp$ takes $\theta\cdot \bH_\Delta$ diagonally into the two copies $\theta_\Delta \cdot \bH_\Delta$ and $(\ba\wedge\bb)\cdot \bH_\Delta$ of $\bH_\Delta$. The result follows.
\end{proof}

\subsection{Computation of $\nu^\nilp$}

Here we assume that $g\ge 4$. Since $N^2=0$, $\bK = M_{-1}\bV^\nilp$. As explained in Section~\ref{sec:boundary_normal}, $\nu^\nilp$ lifts to a section $\nutilde^\nilp$ of $J(M_{-1}\bV^\nilp)$ and projects to a normal function $\nu_\Delta$, which is a section of
$$
J(\Gr^M_{-1}\bV^\nilp) = J(\Lambda^3 \bH_\Delta(-1)).
$$
over $\Delta$.

The first step in understanding $\nu^\nilp$ is to compute the residual normal function $\nu_\Delta$. Since $g-1 \ge 3$ and since
$$
\Lambda^3 \bH_\Delta (-1) \cong \Lambda^3_0 \bH_\Delta(-1) \oplus \bH_\Delta
$$
Proposition~\ref{prop:ext} implies that $\nu_\Delta$ is a linear combination of the normal function $\nu_0$ and $\kappa$ defined there.

\begin{proposition}
\label{prop:residual}
The residual normal function is a linear combination 
$$
\nu_\Delta = \nu_0 + c \kappa \in J(\Lambda^3_0 \bH_\Delta(-1)) \oplus J(\bH_\Delta)
$$
where $c\neq 0$.
\end{proposition}

\begin{proof}[Sketch of proof]
The exact linear combination is easy to compute once one fixes an $\Sp(H_\Delta)$ splitting of $\Lambda^3 \bH_\Delta$. All that matters to us here, though, is that $c$ is non-zero. For this reason, we will not specify a splitting. To establish the result, it is sufficient to work with variations of $\Q$-MHS.

I will assume familiarity with relative completion of mapping class groups \cite{hain:torelli}. Fix $\vv \in L_x'$. We will use it as a base point of $\cM_g$ and $L'$ and $x$ as a base point of $\Delta$. Denote the fibers of $\bE$ and $\bH$ over $\vv$ by $E$ and $H$, and the fibers of $\bE_\Delta$ and $\bH_\Delta$ over $x$ by $E_\Delta$ and $H_\Delta$.

Recall that in this section $S=\cM_g$. Denote the Lie algebra of the completion of $\pi_1(\cM_g,\vv)$ with respect to the standard homomorphism $\pi_1(\cM_g,\vv) \to \Sp(H)$ by $\g_S$. Since the MHS on it is a limit MHS, it has two weight filtrations, $W_\bdot$ and $M_\bdot$. The MHS on $(\g_S,M_\bdot)$ is filtered by $W_\bdot$. Denote the Lie algebra of the relative completion of $\pi_1(L',\vv)$ with respect to the homomorphism $\pi_1(L',\vv) \to \Sp(H_\Delta)$ by $\g_{L'}$. We will denote its weight filtration by $M_\bdot$ instead of $W_\bdot$. We do this because, with this weight filtration, $\g_{L'} \to \g_S$ is a morphism of MHS. Denote the Lie algebra of the relative completion of $\pi_1(\Delta,x)$ with respect $\pi_1(\Delta,x) \to \Sp(H_\Delta)$ by $\g_\Delta$. We will denote its weight filtration by $M_\bdot$. (You can think of it as having two weight filtrations $M_\bdot$ and $W_\bdot$ which are equal.)

The homomorphisms
$$
\begin{tikzcd}
\pi_1(\Delta,x) & \ar[l] \pi_1(L',\vv) \ar[r] & \pi_1(S,\vv)
\end{tikzcd}
$$
induce Lie algebra homomorphisms
$$
\begin{tikzcd}
\g_\Delta & \ar[l] \g_{L'} \ar[r] & \g_S
\end{tikzcd}
$$
which are morphisms of MHS with respect to the weight filtration $M_\bdot$.

The monodromy representations of the normal functions $\nu$, $\nu^\nilp$ and $\nu_\Delta$ are related by the commutative diagram
$$
\begin{tikzcd}
& \pi_1(\cM_{g-1,\vec{2}},\vv) \ar[r]\ar[d,hookrightarrow] &  \pi_1(\cM_{g-1,2},x) \ar[d,hookrightarrow] \\
\pi_1(\cM_g,\vv) \ar[d,"\nu_\ast"] & \ar[l] \pi_1(L',\vv) \ar[r]\ar[d,"(\nu^\nilp)_\ast"] &  \pi_1(\Delta,x) \ar[d,"(\nu_0)_\ast"] \\
W_0\Aut E \ar[r,hookleftarrow] &  A \ar[r,"\bmod M_{-2}"] & M_{-1}\Aut E_\Delta
\end{tikzcd}
$$
where $A$ is the subgroup of $\Aut E$ whose Lie algebra is\footnote{This makes sense as there is (by representation theory of $\sp(H_\Delta)$) a global splitting
$$
\Gr^M_0 \End E = \Q(0) \oplus W_{-1} \Gr^M_0 E
$$
as $W_{-1} \Gr^M_0 E = \Lambda^2_0 H_\Delta (-1)$.}
$$
\fa = \ker\{M_0 \End E \to \End(\Gr^M_0 E) \to W_{-1} \Gr^M_0 E\}.
$$
These induce Lie algebra homomorphisms
\begin{equation}
\label{eqn:comm-diag}
\begin{tikzcd}
\g_S/M_{-3} \ar[d] & \ar[l] \g_{L'}/M_{-3} \ar[r]\ar[d] &  \g_\Delta/M_{-3} \ar[d] \\
W_0\End E \ar[r,hookleftarrow] &  \fa \ar[r,"\mod M_{-2}"] & \End E_\Delta
\end{tikzcd}
\end{equation}
The Lie algebra $\fa$ has a natural MHS induced by that of $\End E$. The homomorphisms in the diagram are morphisms of MHS with respect to the weight filtration $M_\bdot$. Exactness of $\Gr^M_\bdot$ implies that these are determined by the induced maps on the $M_\bdot$ graded quotients. Each $\Gr^M_\bdot \g_X$ is an $\sp(H_\Delta)$ module and all morphisms between them, in this proof, are $\sp(H_\Delta)$ equivariant.

The monodromy of the Ceresa cycle over $\cM_g$ induces (and is determined by) the $\Sp(H)$ equivariant isomorphism
$$
\Gr^W_{-1} \g_S \to \Lambda^3_0 H
$$
(This is the ``Johnson homomorphism''.) It corresponds to the $\sp(H)$ invariant isomorphism
\begin{equation}
\label{eqn:inf-action}
\Gr^W_\bdot \g_S/W_{-2} \overset{\simeq}{\longrightarrow} W_0 \End \Gr^W_\bdot E \cong \sp(H)\ltimes \Lambda^3_0 H
\end{equation}
which is a morphism of MHS with respect to $M_\bdot$. We have to compute its restriction to $\g_{L'}$ and its projection to $\End E_\Delta$.

Proposition~\ref{prop:GrM_V} implies that the $\Gr^M_\bdot\Gr^W_\bdot$ quotients of $\g_S/W_{-2}$ are:
$$
\begin{tikzpicture}
\draw (-5.25,0.5) -- (5.5,0.5);
\draw (-3.75,-1.25) -- (-3.75,1.25);
\matrix[matrix of math nodes,row sep=2mm,column sep=5mm]
{
W_\bdot\bs M_\bdot & -2 & -1 & 0 & 1 & 2 \\
0 & \Q\ba^2 & \ba\cdot H_\Delta & \sp(H_\Delta) & \bb\cdot H_\Delta & \Q \bb^2 \\
-1 & \Lambda^2_0 H_\Delta & \Lambda^3 H_\Delta & \Lambda^2_0 H_\Delta\\
};
\end{tikzpicture}
$$
The bigraded quotients of $E$ are
$$
\begin{tikzpicture}
\draw (-4,0.5) -- (4,0.5);
\draw (-2.3,-1.25) -- (-2.3,1.25);
\matrix[matrix of math nodes,row sep=2mm,column sep=5mm]
{
W_\bdot\bs M_\bdot & -2 & -1 & 0 \\
0 & &  & \Q \\
-1 & \Lambda^2_0 H_\Delta & H_\Delta \oplus \Lambda^3_0 H_\Delta & \Lambda^2_0 H_\Delta\\
};
\end{tikzpicture}
$$
and of $\fa$ are
\begin{equation}
\label{eqn:graded-a}
\begin{tikzpicture}
\draw (-4.8,0.5) -- (4.75,0.5);
\draw (-3,-1.25) -- (-3,1.25);
\matrix[matrix of math nodes,row sep=2mm,column sep=5mm]
{
W_\bdot\bs M_\bdot & -2 & -1 & 0 \\
0 & &  & \sp(H_\Delta) \\
-1 & \Q(1) \oplus \Lambda^2_0 H_\Delta & H_\Delta \oplus \Lambda^3_0 H_\Delta & \\
};
\end{tikzpicture}
\end{equation}
The action of $\g_S/W_{-2}$ on $E$ is determined by the bigraded action
$$
\Gr^M_\bdot \Gr^W_\bdot \g_S \to \End(\Gr^M_\bdot \Gr^W_\bdot E).
$$
The isomorphism (\ref{eqn:inf-action}) implies that the action of $\g_S/W_{-2}$ on $E$ and its associated bigraded version are both faithful. The $\sp(H_\Delta)$ invariance of the action and Schur's Lemma determine the bigraded action up to scaling on each bigraded summand. This will enable us to compute the action of $\Gr^M_\bdot\g_L$ on $\Gr^M_\bdot E$.

Before doing this, note that the summand $\Q\ba^2$ is spanned by the logarithm of the Dehn twist that corresponds to a small loop that encircles $\Delta$. It is the image of the logarithm of the loop in $L_x'$ that encircles the origin.

The $\Gr^M_\bdot$ quotients of the image of $\g_{L'}/M_{-3}$ in $\Gr^M_\bdot \g_S/W_{-3}$, and thus in $\End \Gr^M_\bdot E$ as well, are:
$$
\begin{tikzpicture}
\draw (-5,0) -- (4.5,0);
\draw (-3.65,-.9) -- (-3.65,.9);
\matrix[matrix of math nodes,row sep=2mm,column sep=5mm]
{
M_\bdot & -2 & -1 & 0  \\
& \Q\ba^2 \oplus \Lambda^2_0 H_\Delta  & H_\Delta \oplus \Lambda^3_0 H_\Delta  & \sp(H_\Delta) &  \\
};
\end{tikzpicture}
$$

Since the homomorphism $\g_{L'} \to \g_\Delta$ is surjective, the diagram (\ref{eqn:comm-diag}) and the computation (\ref{eqn:graded-a}) imply that the image of $\g_\Delta \to \End E_\Delta$ is an extension of $\sp(H_\Delta)$ by $H_\Delta \oplus \Lambda^3_0 H_\Delta$. The result now follows from Proposition~\ref{prop:ext}. The copy of  $\Lambda^3_0 H_\Delta$ in the image corresponds to $\nu_0$ and the copy of $H_\Delta$ to $\kappa$.
\end{proof}

\begin{remark}
\label{rem:monodromy}
A consequence of the proof is that the image of $\g_\Delta$ in $\End K$ is an extension of $\sp(H_\Delta)$ by the two step nilpotent Lie algebra whose associated graded is generated by $H_\Delta\oplus \Lambda^3_0 H_\Delta$ in weight $-1$ and whose weight $-2$ graded quotient is $\Lambda^2_0 H_\Delta$. The Lie subalgebra generated by $H_\Delta$ is non-abelian and also has weight $-2$ graded quotient $\Lambda^2_0 H_\Delta$.
\end{remark}

\begin{corollary}
\label{cor:boundary-rk}
We have $\rank \nu_\Delta \ge \rank \nu_0 + 2$.
\end{corollary}

\begin{proof}
Suppose that $x\in \Delta$ is the moduli point of $(C;\{p,q\})$. Since the diagram
$$
\begin{tikzcd}[column sep=scriptsize]
0 \ar[r] & T_p C \oplus T_q C \ar[r] \ar[d,"c(d\kappa|_{C^2})_{\{p,q\}}"] & T_x \Delta \ar[r] \ar[d,"(d\nu_\Delta)_x"] & T_C \cM_{g-1} \ar[r] \ar[d,"(d\nu_0)_C"] & 0\\
0 \ar[r] & T_{\kappa(x)} \Jac C \ar[r] & T_{\nu_\Delta(x)} J(\Lambda^3 H_\Delta) \ar[r] & T_{\nu_0(x)} J(\Lambda^3_0 H_\Delta) \ar[r] & 0
\end{tikzcd}
$$
commutes, where $c$ is the constant appearing in Proposition~\ref{prop:residual}, we see that $\rank \nu_\Delta \ge \rank \nu_0 + \rank \kappa$. The result follows as, for generic $(p,q) \in C^2$, $d\kappa|_{C^2}$ has rank 2 at $(p,q)$.
\end{proof}

\section{Proof of Theorem~\ref{thm:max-rk}}

The first step is to show $\nu^\nilp$ cannot be a section of $J(\bK_\R)$. We do this by restriction to a curve in $\Delta$. Let $C$ be a smooth curve of genus $g-1$ and $p$ a point of $C$. Set $C_p' = C-\{p\}$. There is a morphism $C_p' \to \Delta$ that takes $q \in C$ to the moduli point of the nodal curve $\Cbar_q$ obtained from $C$ by identifying $p$ with $q$. The inclusion $C_p' \hookrightarrow \Delta$ induces a group homomorphism $\pi_1(C_p',q) \to \pi_1(\Delta,x)$, where $x$ is the point of $\Delta$ that corresponds to $\Cbar_q$. This induces a Lie algebra homomorphism $\p(C_p') \to \g_\Delta$ from the Lie algebra of the unipotent completion of $\pi_1(C_p',q)$ to $\g_\Delta$. This homomorphism is an injective morphism of MHS, \cite{hain:torelli}. In particular, $\p(C_p')/W_{-3} \to \g_\Delta/W_{-3}$ is injective. The weight graded quotients of $\p(C_p')/W_{-3}$ are $H_\Delta$ in weight $-1$ and $\Lambda^2 H_\Delta$ in weight $-2$. The bracket $H_\Delta \otimes H_\Delta \to \Lambda^2 H_\Delta$ is surjective.

As mentioned in Remark~\ref{rem:biext}, the section $\nutilde^\nilp$ of $J(\bK)$ over $\Delta$ corresponds to an admissible variation of MHS $\bL$ over $\Delta$. The restriction of $\bL_\Q$ to $C_p'$ is a unipotent variation of MHS with weight graded quotients
$$
\Q,\ \Lambda^3 H_1(C)(-1) \text{ and } \Lambda^2_0 H_1(C).
$$
Denote it by $\bL_C$. In the terminology of \cite{hain:biext}, it is a unipotent biextension. Since its monodromy representation has non-abelian image, the main result of \cite{hain:biext} implies that there is a dense open subset of $C$ over which the fiber of $\bL_C$ does not split as a real biextension. That is, no positive multiple of the restriction of $\nu^\nilp$ to $C$ is a section of $J(\bK_\R)$ and so no positive multiple of $\nu^\nilp$ is a section of $J(\bK_\R)$. Corollary~\ref{cor:splitting} implies that $\rank \nu^\nilp > \rank \nu_\Delta$.

The inductive hypothesis and Corollary~\ref{cor:boundary-rk} imply that
$$
\rank \nu_\Delta \ge \big(3(g-1) -3\big) + 2 = 3g-4.
$$
Since $\rank \nu \ge \rank \nu^\nilp$, we must have $\rank\nu = 3g-3$.

\appendix

\section{Normal functions over $\cM_{h,m+\vec{r}}$}
\label{app:normal}

Denote the category of admissible variations of MHS over a smooth variety $X$ by $\MHS(X)$.
When $\bV$ is a polarized variation of Hodge structure over $X$ of weight $-1$, the space of normal function sections of $J(\bV)$ is, by definition, $\Ext^1_{\MHS(X)}(\Z(0),\bV)$.

Suppose that $2h-2+m+r > 0$. Denote by $\cM_{h,m+\vec{r}}$ the moduli space (more accurately, stack) that parameterizes isomorphism classes of $(C;x_1,\dots,x_m,\vv_1,\dots,\vv_r)$ where $C$ is smooth, projective of genus $h$, each $\vv_j \in T_{y_j} C$ is non-zero, and where $x_1,\dots,x_m,y_1,\dots,y_r$ are distinct points of $C$.

The following is a special case of \cite[Thm.~A.1]{hain:normal} when $r=0$. (See also \cite[\S8]{hain:msri}.) The $r>0$ case follows from the same arguments combined with well-known facts about mapping class groups. Here $\bH$ denotes the local system over $\cM_{h,m+\vec{r}}$ whose fiber over $(C;x_1,\dots,x_m,\vv_1,\dots,\vv_r)$ is $H_1(C)$. For $0 < k \le h$, we set
$$
\Lambda_0^k \bH := \Lambda^k \bH/\big(\theta\cdot \Lambda^{k-2} \bH\big).
$$
It is irreducible and corresponds to the $k$th fundamental representation of $\Sp_h$.

\begin{theorem}
Suppose that $h\ge 3$ and that $\bV$ is a polarized variation of Hodge structure whose underlying local system corresponds to a non-trivial irreducible $\Sp_h$-module. If $\bV$ is not isomorphic to variations $\bH$ or $\Lambda^3_0 \bH(-1)$ of weight $-1$, then
$$
\Ext^1_{\MHS(\cM_{h,m+\vec{r}})}(\Q,\bV) = \Ext^1_{\MHS(\cM_{h,m+r})}(\Q,\bV) = 0.
$$
Otherwise,
$$
\Ext^1_{\MHS(\cM_{h,m+\vec{r})}}(\Q,\bH) = \Ext^1_{\MHS(\cM_{h,m+r})}(\Q,\bH) = \bigoplus_{j=1}^{m+r} \Q\kappa_j
$$
and
$$
\Ext^1_{\MHS(\cM_{h,m+\vec{r})}}(\Q,\Lambda^3_0 \bH(-1)) = \Q\nu.
$$
Here $\nu$ corresponds to the normal function of the Ceresa cycle and $\kappa_j$ corresponds to the section $(C;x_1,\dots,x_m,\vv_1,\dots,\vv_r) \mapsto (2h-2)x_j - K_C$ of the universal jacobian $J(\bH) \to \cM_{g,n+\vec{r}}$, where $x_j = y_{j-m}$ when $j > m$.
\end{theorem}

\begin{remark}
Denote the genus $g$ Torelli group by $T_g$. The Johnson homomorphism \cite{johnson:homom} induces a homomorphism
$$
\tau_g : H_1(T_g;\Z) \to \Lambda^3_0 H_1(C;\Z)
$$
where $C$ is a curve that corresponds to the base point of Torelli space. Johnson \cite{johnson:h1} proved that it is an isomorphism mod 2-torsion. This implies that there are isomorphisms
$$
H^1(\cM_g,\Lambda^3_0\bH_\Q) \cong \Hom(H_1(T_g),\Q) \cong \Q\tau_g
$$
One can show that the generator of $H^1(\cM_g,\Lambda^3_0\bH_\Z)$ corresponds to $2\tau_g$ and that this is the class of the Ceresa cycle. Moreover, the class map
$$
\Ext^1_{\MHS(\cM_g)}(\Z,\Lambda^3_0 \bH_\Z(-1)) \to H^1(\cM_g,\Lambda^3_0\bH_\Z)
$$
is an isomorphism. This implies that when $g\ge 3$, the normal functions (over $\cM_g$) of the Gross--Schoen cycle and all variants of both the Ceresa cycle and the Gross--Schoen cycle are, mod torsion, integral multiples of the normal function of the Ceresa cycle. The class of the Gross Schoen cycle is easily seen to be
$$
(3^3 - 3 \cdot 2^3 + 3 \cdot 1^3)\tau_g = 6 \tau_g,
$$
which implies that the normal function of the Gross--Schoen cycle is 3 times that of the Ceresa cycle.
\end{remark}

\end{document}